\documentclass[11pt,reqno]{amsart}
\usepackage{amsmath,amssymb,mathrsfs,color,amsthm}
\usepackage[title]{appendix}
\usepackage{enumitem}
\usepackage{xcolor}
\usepackage[colorlinks=true,linkcolor=blue,citecolor=red,urlcolor=cyan]{hyperref}

\makeatother
\let\cal=\mathcal
\def\N{{\mathbb N}}

\def\R{{\mathbb R}}
\def\P{{\mathbb P}}
\def\E{{\mathbb E}}
\def\T{{\mathbb T}}

\def\var{\varepsilon}

\newtheorem{thm}{Theorem}[section]
\newtheorem{cor}[thm]{Corollary}
\newtheorem{lem}[thm]{Lemma}
\newtheorem{prop}[thm]{Proposition}

\theoremstyle{definition}
\newtheorem{de}[thm]{Definition}
\theoremstyle{remark}
\newtheorem{rem}[thm]{Remark}
\newtheorem{exam}[thm]{Example}
\numberwithin{equation}{section}
\newcommand{\Law}[1]{\mathcal{L}_{#1}}
\newcommand{\norm}[1]{\left\| #1 \right\|}
\newcommand{\inpro}[1]{\left\langle #1 \right\rangle}

\allowdisplaybreaks

\begin{document}

\title[Weak Averaging Principle and Weak Pullback Attractors]{Weak Averaging Principle and Weak Pullback Attractors for Mckean--Vlasov Stochastic Navier--Stokes Equations}

\author{Honglei Chen}
\address{H. Chen: School of Mathematical Sciences, Dalian University of Technology, Dalian
116024, P. R. China}
\email{hcl9712285@gmail.com}

\author{Zhenxin Liu}
\address{Z. Liu: School of Mathematical Sciences, Dalian University of Technology, Dalian
116024, P. R. China}
\email{zxliu@dlut.edu.cn}
   
\date{September 14, 2026}  



\begin{abstract}
We establish three weak averaging principles for distribution-dependent stochastic Navier--Stokes equations with rapidly oscillating coefficients on the torus $\T^2$. 
First, solution laws are relatively compact on finite time intervals, and every limit point as $\var\to0$ is the path law of a variational martingale solution of the averaged equation.
Under a dissipative condition, we can choose bounded complete
variational solution laws of the original and averaged equations so that
every sequence with $\var\to0$ has a subsequence converging weakly to a
bounded complete variational solution law of the averaged equation.
At the level of probability laws, the original nonautonomous equation admits a family of weak pullback attractors, while the averaged equation has a weak global attractor.
The weak pullback attractors converge upper semicontinuously to the weak global attractor of the averaged equation, uniformly with respect to the coefficient hull.
\end{abstract}

\maketitle

\section{Introduction}

In this paper, we study weak averaging principles and asymptotic dynamics for a
class of distribution-dependent stochastic two-dimensional Navier--Stokes
equations with rapidly oscillating coefficients on the torus \(\T^2\).
For \(0<\var\ll1\), the equation under consideration is
\begin{equation}\label{eq:SPDEone}
\begin{aligned}
du_\var(t)
+\left[
\nu Au_\var(t)+B(u_\var(t),u_\var(t))
\right]dt
=
f\left(\frac{t}{\var},u_\var(t),\Law{u_\var(t)}\right)dt
+
g\left(\frac{t}{\var},u_\var(t),\Law{u_\var(t)}\right)dW_t,
\end{aligned}
\end{equation}
where \(\Law{u_\var(t)}\) denotes the law of \(u_\var(t)\),
\(W_t\) is a cylindrical Wiener process on a separable Hilbert space \(U\),
$f:\R\times H\times\mathcal P_2(H)\to H$,
$g:\R\times H\times\mathcal P_2(H)\to L_2(U,H)$,
and \(L_2(U,H)\) is the space of Hilbert--Schmidt operators from \(U\) to
\(H\). Here \(H\) is the zero-mean divergence-free subspace of
\(L^2(\T^2;\R^2)\) defined in \eqref{99:1}, and \(\mathcal P_2(H)\) is the
space of Borel probability measures on \(H\) with finite second moments.
Moreover, \(A\) is the Stokes operator, \(B\) is the Navier--Stokes
bilinear operator, and \(\nu>0\) is the kinematic viscosity coefficient.
The dependence of \(f\) and \(g\) on \(\Law{u_\var(t)}\) represents the
mean-field interaction of the system.

Averaging principles are fundamental tools for the analysis of dynamical
systems with separated time scales. They show that rapidly oscillating
systems can be approximated by effective systems in which the fast
oscillations are averaged out. The classical theory goes back to Krylov,
Bogolyubov and Mitropolsky
\cite{bogolyubov1961asymptotic,krylov1943introduction}, and was later
extended to stochastic differential equations by Khasminskii
\cite{has1966stochastic,khasminskij1968principle,khasminskii2004averaging}.
Since then, stochastic averaging has been extensively developed for fully
coupled systems, time-dependent coefficients, fractional noise, normal
deviations and large deviations; see, for example,
\cite{bakhtin2004diffusion,freidlinWentzell1998,kifer2004some,
veretennikov1991averaging,veretennikov1999large,liu2020averaging,
hairer2020averaging,pei2021averaging,cheng2023averaging2}.
Long-time and uniform-in-time averaging, including periodic and almost-periodic settings, has been studied in
\cite{freidlin2006long,kamenskii2015weak,cheban2021averaging,
liuwang2021averaging,uda2021averaging,crisan2026uniform,xie2026uniform}.

The averaging principle has also been extensively developed for stochastic
partial differential equations under various structural assumptions and
noise regimes; see, for example,
\cite{cerrai2009khasminskii,cerrai2009averaging,cerrai2011averaging,
brehier2012strong,wang2012average,duan2014effective,fu2015strong,
bao2017twotime,pei2017stochastic,cerrai2017averaging,
dong2018burgers,gao2019bogoliubov}
and the more recent works
\cite{sun2021averaging,hong2022mckean,cheng2023second,
cheng2023averaging,defeo2023order,MR2023Asymptotic,
cheng2026nonautonomous}, together with the references therein.
In particular, averaging principles for stochastic two-dimensional
Navier--Stokes equations have been studied in
\cite{gao2021averaging,gao2022averaging,gao2023averaging}.

Stochastic Navier--Stokes equations constitute a fundamental class of
stochastic fluid models, and their well-posedness, stationary and ergodic
behavior, and long-time dynamics have been extensively studied; see, for
example,
\cite{bensoussan1995stochastic,flandoli1995martingale,
flandoli1995ergodicity,kuksin2012math,brzezniak2013random,
brzezniak2018random}.
On the other hand, McKean--Vlasov equations, originating from the work of
McKean \cite{mckean1966class}, describe mean-field interactions through the
dependence of the coefficients on the law of the solution and are closely
related to propagation of chaos and mean-field models; see
\cite{sznitman1991topics,buckdahn2017meanfield,lasry2007mean,
carmona2018probabilistic}. Their stochastic infinite-dimensional
counterparts have also attracted increasing attention
\cite{hong2024mckean}. 
Averaging principles for McKean--Vlasov equations
have subsequently been established in both finite- and infinite-dimensional
settings
\cite{rockner2021mckean,xu2021twotime,hong2022mckean,
shen2022averaging,shen2025twotime,cheng2024ddsde,shi2025averaging,
yin2026stability}, while distribution-dependent stochastic fluid models
have recently been considered in \cite{chen2025homogenization}.
Related long-time dynamics have been studied for distribution-dependent
stochastic Navier--Stokes equations in \cite{zhang2026uniform}, and
averaging together with attractor convergence for distribution-dependent
stochastic reaction--diffusion equations was investigated in
\cite{chen2026averaging}. These developments motivate us to study the
averaging behavior and asymptotic dynamics of \eqref{eq:SPDEone}
simultaneously. Our main results are described below.

First, we establish a finite-time weak Bogolyubov averaging principle for the
distribution-dependent stochastic two-dimensional Navier--Stokes equation
\eqref{eq:SPDEone}. More precisely, if \(u_\var(t)\), \(0<\var\le1\), are
variational martingale solutions of \eqref{eq:SPDEone} with
\(\Law{u_\var(0)}=\Law{u_0}\) for some
\(u_0\in L^4(\Omega,\mathcal F_0,\mathbb P;V)\), then, for every \(T>0\),
\[
\{\Law{u_\var|_{[0,T]}}:0<\var\le1\}
\text{ is relatively compact in }\mathcal P(C([0,T];H)),
\]
and every limit point as \(\var\to0\) is the path law of a variational
martingale solution of the averaged equation
\begin{equation}
\label{eq:SPDEtwo}
\begin{aligned}
d\bar u(t)
+\left[
\nu A\bar u(t)+B(\bar u(t),\bar u(t))
\right]dt
=
\bar f\left(\bar u(t),\Law{\bar u(t)}\right)dt
+
\bar g\left(\bar u(t),\Law{\bar u(t)}\right)dW_t,
\end{aligned}
\end{equation}
with \(\Law{\bar u(0)}=\Law{u_0}\).
Here the averaged coefficients \(\bar f\) and \(\bar g\) are specified in
\ref{item:H-a}. Related finite-time averaging results for stochastic
two-dimensional Navier--Stokes equations can be found in
\cite{gao2021averaging,gao2022averaging,gao2023averaging}, while homogenization
of distribution-dependent stochastic fluid models was studied in
\cite{chen2025homogenization}. Averaging in the sense of convergence in
distribution for stochastic evolution equations was established in
\cite{kamenskii2015weak}. The result above establishes the corresponding
weak averaging principle for \eqref{eq:SPDEone}, where both the drift and
diffusion depend on the law of the solution, without assuming uniqueness
in law.

Second, under an additional dissipativity condition, we establish a weak
second Bogolyubov theorem for \eqref{eq:SPDEone}.
Here ``weak'' refers to subsequential weak convergence of complete solution
laws, without requiring uniqueness.
For every \(0<\var\le1\), the original equation admits a bounded complete
variational solution law, and the averaged equation \eqref{eq:SPDEtwo}
also admits such a solution law. These solution laws can be chosen with
uniform second-moment, \(V\)-moment, and higher-moment bounds.
Moreover, for every sequence \(\var_n\to0\),
\[
\{\Law{u_{\var}(\cdot)}\}_{n\ge1}
\text{ is relatively compact in }\mathcal P(C_{\mathrm{loc}}(\R;H)),
\]
and every subsequential limit
is a bounded complete variational solution law of the averaged equation
\eqref{eq:SPDEtwo}. Thus, the second averaging result extends compactness
and limit identification from finite time intervals to complete
variational solution laws on the whole real line.

Second Bogolyubov theorems and whole-line averaging for
distribution-independent stochastic equations have been studied in
\cite{cheban2021averaging,cheng2023averaging}. Under conditions ensuring
uniqueness of whole-line \(L^2\)-bounded solutions,
\cite{cheng2023second} establishes weak convergence of variational
solution laws on continuous path space, while
\cite{liuwang2021averaging} obtains convergence of time-marginal
distributions uniformly on the whole real line.
For distribution-dependent stochastic
reaction--diffusion equations, a stronger whole-line mean-square
convergence can be obtained under an appropriate contraction structure
\cite{chen2026averaging}. In the present Navier--Stokes setting, the
dissipativity condition does not ensure uniqueness of bounded complete
variational solution laws; see Example \ref{no-uniqueness}.
Such a whole-line mean-square conclusion is not asserted here:
instead, we identify all possible limit points of the selected bounded
complete variational solution laws in
\(\mathcal P(C_{\mathrm{loc}}(\R;H))\).

Third, we study the long-time behavior of the probability laws generated by
\eqref{eq:SPDEone} and \eqref{eq:SPDEtwo}. Since the coefficients of
\eqref{eq:SPDEone} depend on time, its dynamics is nonautonomous and is
considered over the hull \(\mathcal H(F_0)\), namely the closure of the time
translates of \(F_0=(f,g)\).
We prove that the original equation admits a weak pullback attractor
\(\{\mathcal A^\var(F)\}_{F\in\mathcal H(F_0)}\), whereas the autonomous
averaged equation admits a weak global attractor
\(\bar{\mathcal A}\subset\mathcal P_2(H)\).
Here attraction is taken over all variational solution laws with initial
laws in the sets \(K\in\cal D_p\), as specified in Definition
\ref{def:weak-pullback-attractor-NS}. Moreover,
\[
\lim_{\var\to0}
\sup_{F\in\mathcal H(F_0)}
\sup_{\mu\in\mathcal A^\var(F)}
\inf_{\nu\in\bar{\mathcal A}}
d_{BL}(\mu,\nu)
=0.
\]
Hence, the weak pullback attractors converge upperm semicontinuously to the
weak global attractor of the averaged equation, uniformly with respect to
\(F\in\mathcal H(F_0)\).

The main difficulties arise from the Navier--Stokes nonlinearity, the
possible nonuniqueness of bounded complete variational solution laws, and
the uniform estimates needed for attractor convergence. The first difficulty
is that the local comparison estimate for \(B(u,u)\) depends on the
\(V\)-norms of the solutions, while stopping the trajectories does not replace
the marginal laws in the coefficients by those of the stopped processes.
Consequently, the usual localization argument does not directly provide a
closed mean-square stability estimate. 
We overcome this difficulty by combining uniform a priori estimates,
path tightness, and a martingale characterization of the limit,
without relying on uniqueness in law. 
A further difficulty
comes from the construction of bounded complete variational solution laws.
Since these laws are obtained through pullback limits and need not be unique,
we establish their averaging behavior through \(W_2\)-compactness of the time
marginals, finite-window averaging, and consistency of the limiting path laws.
Finally, to prove the upper semicontinuity of the weak attractors, we use
pullback absorption together with uniform \(V\)- and \(2p\)-moment estimates
to obtain \(W_2\)-compactness of the initial laws of finite solution segments
whose terminal laws approximate attractor elements. Applying the averaging
principle on these segments and then the attraction property of the averaged
equation yields the desired upper semicontinuity in the \(d_{BL}\)-topology,
uniformly with respect to the coefficient hull.

Our results are closely related to recent work on averaging and long-time
dynamics for stochastic Navier--Stokes and distribution-dependent equations
\cite{gao2021averaging,gao2022averaging,gao2023averaging,
chen2025homogenization,zhang2026uniform,chen2026averaging}. In particular,
\cite{zhang2026uniform} studies uniform measure attractors for
distribution-dependent stochastic two-dimensional Navier--Stokes equations,
whereas we consider the convergence of weak pullback attractors under
averaging. The work \cite{chen2026averaging} is closest to ours in its
overall objective, but concerns stochastic reaction--diffusion equations
and obtains whole-line mean-square averaging under a contraction condition.
In the present Navier--Stokes setting, neither uniqueness in law for
variational martingale solutions nor uniqueness of bounded complete
variational solution laws is assumed, and both the finite-time and
whole-line results are formulated through compactness and identification
of subsequential limits.

The paper is organized as follows. In Section 2, we introduce the
assumptions and main results. In Section 3 and Section 4, we prove the
weak first Bogolyubov theorem on finite time intervals and establish the
weak second Bogolyubov theorem for bounded complete variational solution
laws of \eqref{eq:SPDEone}, respectively. In Section 5, we study the
long-time behavior of all variational solution laws, prove the existence
of the weak pullback attractors and the weak global attractor of the
averaged equation, and establish the upper semicontinuous convergence
of the former to the latter.

\subsection*{Notation}

Throughout this paper, we work on the two-dimensional torus $\T^2$.
Let $\Pi$ be the Leray projection from $L^2(\T^2;\R^2)$ onto the space of divergence-free vector fields. We define
\begin{align}\label{99:1}
H&:=\left\{u\in L^2(\T^2;\R^2):\operatorname{div}u=0,\ 
\int_{\T^2}u(x)\,dx=0\right\},
&V&:=H^1(\T^2;\R^2)\cap H.
\end{align}
Throughout the paper, all drift and diffusion terms are understood as $H$-valued terms. If the original coefficients are not divergence-free, they are replaced by their Leray projections. Namely, we write $f$ and $g$ for $\Pi f$ and $\Pi g$ whenever necessary.
We define
\[
\begin{aligned}
\norm{u}^2&:=\int_{\T^2}|u(x)|^2\,dx,
&\inpro{u,v}&:=\int_{\T^2}u(x)\cdot v(x)\,dx,\\
\norm{u}_V^2&:=\int_{\T^2}|\nabla u(x)|^2\,dx,
&\inpro{u,v}_V&:=\int_{\T^2}\nabla u(x):\nabla v(x)\,dx.
\end{aligned}
\]
Whenever \(\norm{\cdot}_V\) is considered on \(H\), we extend it by
\[
\norm{u}_V:=+\infty,\qquad u\in H\setminus V.
\]
Let $V^*$ be the dual of $V$ with pivot space $H$; the brackets $\inpro{\cdot,\cdot}$ also denote the duality pairing between $V^*$ and $V$. Let
\[
A:=-\Pi\Delta,\qquad D(A):=H^2(\T^2;\R^2)\cap V
\]
be the Stokes operator.
For $u,v\in V$, define
\[
B(u,v):=\Pi((u\cdot\nabla)v).
\]
We also write
\[
b(u,v,w):=\int_{\T^2}(u\cdot\nabla)v\cdot w\,dx,
\qquad u,v,w\in V.
\]
Then $\inpro{B(u,v),w}=b(u,v,w)$.
On the zero-mean subspace, the Poincar\'e inequality holds with a constant $\lambda_1>0$:
\[
\norm{u}^2\le \frac{1}{\lambda_1}\norm{u}_V^2,\qquad u\in H.
\]
Let $\cal P(H)$ be the space of Borel probability measures on $H$, and let $\cal P_2(H)$ consist of those with finite second moments. 
We equip $\cal P_2(H)$ with the Wasserstein distance $W_2$, defined by
\[W_2^2(\mu,\nu):=\inf_{\pi\in\Pi(\mu,\nu)}\int_{H\times H}\|x-y\|^2\,\pi(dx,dy),\]
where $\Pi(\mu,\nu)$ is the set of probability measures on $H\times H$ with marginals $\mu$ and $\nu$.
Equivalently, for any random variables $X,Y$ defined on a common probability space with
$\Law{X}=\mu$ and $\Law{Y}=\nu$,
\[
W_2^2(\mu,\nu)=\inf \E\norm{X-Y}^2,
\]
where the infimum is taken over all such couplings.
Similarly, $W_{2,V}$ denotes the Wasserstein distance on $\cal P_2(V)$ induced by the norm $\norm{\cdot}_V$.
Let $\delta_0$ be the Dirac measure at $0\in H$. Let \((E,d_E)\) be a metric space. For a bounded Lipschitz function
\(\phi:E\to\R\), define
\[
\norm{\phi}_{BL}:=\norm{\phi}_{\infty}+\mathrm{Lip}(\phi),
\]
where
\[
\norm{\phi}_{\infty}:=\sup_{x\in E}|\phi(x)|,
\qquad
\mathrm{Lip}(\phi):=
\sup_{x\neq y}\frac{|\phi(x)-\phi(y)|}{d_E(x,y)}.
\]
For \(\mu,\nu\in\mathcal P(E)\), the bounded-Lipschitz distance is defined by
\[
d_{BL}(\mu,\nu)
:=
\sup_{\norm{\phi}_{BL}\le1}
\left|
\int_E\phi(x)\,\mu(dx)
-
\int_E\phi(x)\,\nu(dx)
\right|.
\]
The notation $\Rightarrow$ denotes weak convergence of probability measures. We equip $C([s,t];H)$ with the supremum norm.
We denote by \(C_{\mathrm{loc}}(\R;H)\) the space \(C(\R;H)\) endowed with the
compact-open topology, namely,
\[
x_n\to x \quad\text{in } C_{\mathrm{loc}}(\R;H)
\]
if and only if, for every \(R\in\mathbb N\),
\[
\sup_{t\in[-R,R]}\norm{x_n(t)-x(t)}\to0 .
\]
Let $\lfloor C\rfloor$ denote the integer part of $C$ for any $C\geq0$.
We use $C$ with or without subscripts to denote some constant, 
which may change from line to line.

\section{Statement of Main Results}
Let $(\Omega,\cal F,\{\cal F_t\}_{t\in\R},\P)$ be a filtered
probability space satisfying the usual conditions, and let
$\{W_t\}_{t\in\R}$ be a two-sided cylindrical Wiener process on
a separable Hilbert space $U$ such that, for every $s\in\R$,
$\{W_{s+t}-W_s\}_{t\ge0}$ is a cylindrical Wiener process with
respect to $\{\cal F_{s+t}\}_{t\ge0}$.
We consider the following distribution-dependent stochastic Navier--Stokes equation on $\T^2$:
\[
\begin{aligned}
du_\var(t)
+\left[
\nu A u_\var(t)+B(u_\var(t),u_\var(t))
\right]dt
=
f\left(\frac{t}{\var},u_\var(t),\Law{u_\var(t)}\right)dt
+
g\left(\frac{t}{\var},u_\var(t),\Law{u_\var(t)}\right)dW_t,
\end{aligned}
\]
where $\nu>0$ is the kinematic viscosity coefficient, $f:\R\times H\times\cal P_2(H)\to H$ and $g:\R\times H\times\cal P_2(H)\to L_2(U,H)$ are measurable mappings.

\subsection{Conditions and Assumptions}\label{sec:assump}
We have the following standard identities:
\begin{align}
\label{bskew}
b(u,v,w)&=-b(u,w,v),\\
\label{bzero}
b(u,v,v)&=0,\\
\label{BzeroH}
\inpro{B(u,u),u}&=0,\\
\label{BzeroV}
\inpro{B(u,u),Au}&=0,\qquad u\in D(A),
\end{align}
where \eqref{BzeroV} is specific to the two-dimensional incompressible periodic case.
Moreover, for $u,v,w\in V$, there exists a constant $C$ such that
\begin{align}
\label{bloc}
|b(w,v,w)|
\le C\norm{v}_V\norm{w}\norm{w}_V.
\end{align}
Assume the following conditions hold.

\begin{enumerate}[label=\textup{(H\arabic*)},leftmargin=2.7em]

\item\label{item:H-f}
There exist constants $K_f,L_f>0$ such that for all $t\in\R$, $x,y\in H$, and $\mu_1,\mu_2\in\cal P_2(H)$,
\[
\norm{f(t,0,\delta_0)}\le K_f,\quad
\norm{f(t,x,\mu_1)-f(t,y,\mu_2)}
\le
L_f\left(\norm{x-y}+W_2(\mu_1,\mu_2)\right).
\]

\item\label{item:H-g}
(i) There exist constants $K_g>0$ and $L_g\ge0$ such that for all $t\in\R$, $x,y\in H$, and $\mu_1,\mu_2\in\cal P_2(H)$,
\[
\norm{g(t,0,\delta_0)}_{L_2(U,H)}\le K_g,
\]
\[
\norm{g(t,x,\mu_1)-g(t,y,\mu_2)}_{L_2(U,H)}
\le
L_g\left(\norm{x-y}+W_2(\mu_1,\mu_2)\right);
\]
(ii) There exist constants $K_g'>0$ and $L_g'\ge0$ such that for all $t\in\R$, $x,y\in V$, and $\mu,\mu_1,\mu_2\in\cal P_2(V)$,
\[
g(t,x,\mu)\in L_2(U,V),
\quad
\norm{g(t,0,\delta_0)}_{L_2(U,V)}\le K_g',
\]
and
\[
\norm{g(t,x,\mu_1)-g(t,y,\mu_2)}_{L_2(U,V)}
\le
L_g'\left(\norm{x-y}_V+W_{2,V}(\mu_1,\mu_2)\right).
\]

\item\label{item:H-a}
There exists a function $\omega^f:\R_+\to\R_+$ satisfying $\omega^f(T)\to0$ as $T\to\infty$, 
such that for all $x\in H$ and $\mu\in\cal P_2(H)$,
\begin{align}
\label{wf}
\sup_{a\in\R}
\norm{
\frac1T\int_a^{a+T}
\left[
f(s,x,\mu)-\bar f(x,\mu)
\right]\,ds
}
\le
\omega^f(T)
\left(
1+\norm{x}+\left(\mu(\norm{\cdot}^2)\right)^{1/2}
\right),
\end{align}
where $\mu(\norm{\cdot}^2):=\int_H\norm{x}^2\,\mu(dx)$.

There exists a function $\omega^g:\R_+\to\R_+$ satisfying
$\omega^g(T)\to0$ as $T\to\infty$,
such that for all $x\in H$ and $\mu\in\cal P_2(H)$,
\begin{align}
\label{wg}
\sup_{a\in\R}
\frac1T\int_a^{a+T}
\norm{
g(s,x,\mu)-\bar g(x,\mu)
}_{L_2(U,H)}^2\,ds
\le
\omega^g(T)
\left(
1+\norm{x}^2+\mu(\norm{\cdot}^2)
\right).
\end{align}

\item\label{item:H-d}
There exist constants \(c_0\ge0\), \(c_1\in\R\), and \(c_2\ge0\) such that for all
\(t\in\R\), \(u\in H\), and \(\mu\in\cal P_2(H)\),
\begin{align}
\label{Hfour}
2\inpro{f(t,u,\mu),u}
+
\norm{g(t,u,\mu)}_{L_2(U,H)}^2
\le
c_0+c_1\norm{u}^2+c_2\mu(\norm{\cdot}^2).
\end{align}
\end{enumerate}

\begin{rem}
\label{rem:averaged-coefficients-NS}
(i) 
By \ref{item:H-f}, \ref{item:H-g} and Young's inequality,
we have for any $\iota\in(0,1)$, all $t \in \R$, $x \in H$ and $\mu \in \cal{P}_2(H)$,
\begin{align}\label{fcon}
\norm{f\left(t,x,\mu\right)}
\leq L_f\|x\|+L_f\mu(\|\cdot\|^2)^{\frac12}+K_f,
\end{align}
\begin{align}\label{gcon}
\norm{g\left(t,x,\mu\right)}_{L_2(U,H)}^2
\leq (L_g^2+\iota)\|x\|^2+C_\iota\left(\mu(\|\cdot\|^2)+1\right).
\end{align}

(ii) Under conditions \ref{item:H-f}, \ref{item:H-g} and \ref{item:H-a},
the averaged coefficients $\bar f$ and $\bar g$ satisfy \ref{item:H-f},
\ref{item:H-g}, \eqref{fcon} and \eqref{gcon}. For $\bar g$, the $V$-valued part of \ref{item:H-g} follows from
\ref{item:H-g}(ii) and \eqref{wg} by weak compactness in $L_2(U,V)$
and weak lower semicontinuity.

(iii) By Lemma \ref{lem:mild-solution}, under
\ref{item:H-f}, \ref{item:H-g} and \ref{item:H-a}, equations \eqref{eq:SPDEone}
and \eqref{eq:SPDEtwo} admit variational martingale solutions
in the sense of Definition \ref{def:variational-martingale-solution-NS}
on $[s,\infty)$ for every \(s\in\R\) and \(0<\var\le1\).
The initial law may be any \(\mu_0\in\cal P_2(H)\) satisfying
\[\int_H\norm{x}^{2p}\,\mu_0(dx)<\infty \text{ for some } p>1.\]
\end{rem}

\begin{rem}\label{rem:bar}
If \ref{item:H-d} holds, then $\bar f$ and $\bar g$
also satisfy the same dissipativity condition, namely,
\[
2\inpro{\bar f(u,\mu),u}
+
\norm{\bar g(u,\mu)}_{L_2(U,H)}^2
\le
c_0+c_1\norm{u}^2+c_2\mu(\norm{\cdot}^2),
\qquad u\in H,\ \mu\in\cal P_2(H).
\]
Indeed, this follows by averaging the inequality in \ref{item:H-d} and
using Jensen's inequality for the term $\|\bar g(u,\mu)\|_{L_2(U,H)}^2$.
\end{rem}

\subsection{Main results}
Our first result is the finite-time weak Bogolyubov theorem.
\begin{thm}
\label{thmT}
Assume that \ref{item:H-f}, \ref{item:H-g} and \ref{item:H-a} hold.
For $0<\var\le1$, let $u_\var(t)$ be any variational martingale solution
of \eqref{eq:SPDEone} in the sense of Definition
\ref{def:variational-martingale-solution-NS}.
Assume that $\Law{u_\var(0)}=\Law{u_0}$ for some
$u_0\in L^4(\Omega,\cal F_0,\P;V)$.
Then, for every $T>0$, the family
\[\{\Law{u_\var|_{[0,T]}}:0<\var\le1\}\text{ is relatively compact in }
\cal P(C([0,T];H)).\]
Every limit point as $\var\to0$ is the path law of a variational
martingale solution $\bar u$ of \eqref{eq:SPDEtwo} on $[0,T]$
with $\Law{\bar u(0)}=\Law{u_0}$.
\end{thm}

The following result establishes the existence of bounded complete variational solution laws in the sense of Definition \ref{def:complete-variational-solution-law-NS} below and identifies their limit points under the averaging limit.
\begin{thm}
\label{thmR}
Assume that \ref{item:H-f}, \ref{item:H-g}, \ref{item:H-a} and
\ref{item:H-d} hold. Suppose further that
\(2\nu\lambda_1>c_1+c_2\). Then the following conclusions hold.
\begin{enumerate}
\item For every \(0<\var\le1\), equation \eqref{eq:SPDEone} admits a
bounded complete variational solution law in the sense of Definition
\ref{def:complete-variational-solution-law-NS}. For some \(p>1\),
these solution laws can be chosen so that
\begin{equation}
\label{eq:thmR-original-bound-NS}
\sup_{0<\var\le1}\sup_{t\in\R}
\E\norm{u_\var(t)}^2
+
\sup_{0<\var\le1}\sup_{t\in\R}
\E\norm{u_\var(t)}_V^2
+
\sup_{0<\var\le1}\sup_{t\in\R}
\E\norm{u_\var(t)}^{2p}
<\infty.
\end{equation}

\item Equation \eqref{eq:SPDEtwo} admits a bounded complete variational
solution law in the sense of Definition
\ref{def:complete-variational-solution-law-NS}. For the same \(p>1\),
this solution law can be chosen so that
\begin{equation}
\label{eq:thmR-averaged-bound-NS}
\sup_{t\in\R}
\E\norm{\bar u(t)}^2
+
\sup_{t\in\R}
\E\norm{\bar u(t)}_V^2
+
\sup_{t\in\R}
\E\norm{\bar u(t)}^{2p}
<\infty.
\end{equation}

\item 
For every sequence \(\var_n\to0\), the path laws
\(\Law{u_{\var_n}(\cdot)}\) form a relatively compact family in
\(\cal P(C_{\mathrm{loc}}(\R;H))\).
Every subsequential limit is a bounded complete variational solution
law of the averaged equation \eqref{eq:SPDEtwo} in the sense of
Definition \ref{def:complete-variational-solution-law-NS}.
\end{enumerate}
\end{thm}

We next formulate the global averaging principle in the weak sense
for the attractors in Definition \ref{def:weak-pullback-attractor-NS}.
The main result concerns their upper semicontinuity under averaging,
uniformly over the hull.

Write \(F_0=(f,g)\) and let
\[
\sigma_\tau F_0:=(f(\cdot+\tau,\cdot,\cdot),g(\cdot+\tau,\cdot,\cdot)),
\qquad \tau\in\R.
\]
For \(l\in\N\), set
\[
K_l
:=
\left\{
\mu\in\cal P_2(H):
\int_H\norm{z}^2\,\mu(dz)\le l^2
\right\},
\]
and
\begin{align*}
d_l(F_1,F_2)
:=
\sup_{\substack{|s|\le l,\ \norm{x}\le l\\ \mu\in K_l}}
\big(
\norm{f_1(s,x,\mu)-f_2(s,x,\mu)}+
\norm{g_1(s,x,\mu)-g_2(s,x,\mu)}_{L_2(U,H)}
\big).
\end{align*}
We equip the coefficient space with the metric
\[
d(F_1,F_2)
:=
\sum_{l=1}^\infty
2^{-l}
\frac{d_l(F_1,F_2)}{1+d_l(F_1,F_2)}
\]
and define the hull
\[
\cal H(F_0)
:=
\overline{
\{\sigma_\tau F_0:\tau\in\R\}
}^d.
\]
\begin{thm}
\label{thm:global-weak-attractor-averaging-NS}
Assume that \ref{item:H-f}, \ref{item:H-g}, \ref{item:H-a}, and
\ref{item:H-d} hold with $2\nu\lambda_1>c_1+c_2$.
Then there exists $p>1$ such that, with
\[
\cal D_p
:=
\left\{
K\subset\cal P_2(H):
\sup_{\mu\in K}\int_H\norm{x}^{2p}\,\mu(dx)<\infty
\right\},
\]
the following conclusions hold.
\begin{enumerate}
\item For every $0<\var\le1$, equation \eqref{eq:SPDEone}
admits a weak pullback attractor
$\{\cal A^\var(F)\}_{F\in\cal H(F_0)}$
with respect to $\cal D_p$ in the sense of Definition
\ref{def:weak-pullback-attractor-NS}.

\item Equation \eqref{eq:SPDEtwo} admits a weak global attractor
$\bar{\cal A}$ with respect to $\cal D_p$ in the sense of
Definition \ref{def:weak-pullback-attractor-NS}.

\item The attractors converge in the sense that
\[
\lim_{\var\to0}
\sup_{F\in\cal H(F_0)}
\sup_{\mu\in\cal A^\var(F)}
\inf_{\nu\in\bar{\cal A}}
d_{BL}(\mu,\nu)=0.
\]
\end{enumerate}
\end{thm}

\section{The Weak First Bogolyubov Theorem}
In this section, we establish the weak first Bogolyubov theorem for \eqref{eq:SPDEone} on finite intervals. 
We work with variational martingale solutions, where the stochastic
basis and the driving Wiener process are regarded as part of the solution; 
see \cite{flandoli1995martingale,hong2024mckean}.

\begin{de}\label{def:variational-martingale-solution-NS}
Let \(\mu_0\in\cal P_2(H)\). A \emph{variational martingale solution}
of \eqref{eq:SPDEone} on \([s,t]\) is an adapted process \(u_\var\)
defined on a stochastic basis carrying a cylindrical Wiener process
\(W\) on \(U\), with \(\Law{u_\var(s)}=\mu_0\).
Moreover,
\[
u_\var\in L^2\bigl(\Omega;C([s,t];H)\bigr)
\cap L^2\bigl(\Omega;L^2(s,t;V)\bigr).
\]
For every \(\phi\in V\), almost surely for all \(r\in[s,t]\),
\[
\begin{aligned}
\inpro{u_\var(r),\phi}
&=\inpro{u_\var(s),\phi}
-\int_s^r
\left[
\nu\inpro{u_\var(q),\phi}_V
+b(u_\var(q),u_\var(q),\phi)
\right]dq\\
&\quad+
\int_s^r
\inpro{
f\left(\frac q\var,u_\var(q),\Law{u_\var(q)}\right),
\phi
}\,dq\\
&\quad+
\int_s^r
\inpro{
g\left(\frac q\var,u_\var(q),\Law{u_\var(q)}\right)dW_q,
\phi
}.
\end{aligned}
\]
The path law of \(u_\var\) is called a \emph{variational solution law}.
The same definition applies to \eqref{eq:SPDEtwo}, with \(f\) and \(g\)
replaced by \(\bar f\) and \(\bar g\).
\end{de}
We do not assume uniqueness in law for variational martingale solutions.

\subsection{A priori and continuity estimates}

\begin{lem}
\label{lemone}
Assume that \ref{item:H-f} and \ref{item:H-g}(i) hold.
Let $u_\var(t)$ be any variational martingale solution of
\eqref{eq:SPDEone}, with initial data $u_0$.
For any $p\ge1$ and $T>0$, if
$u_0\in L^{2p}(\Omega,\cal F_0,\P;H)$, then
\begin{equation}
\label{eq:SPDEthree}
\begin{aligned}
&\E\sup_{0\le t\le T}\norm{u_\var(t)}^{2p}
+
\E\int_0^T
\norm{u_\var(t)}^{2p-2}\norm{u_\var(t)}_V^2\,dt\\
&\qquad\le
C_T\left(1+\E\norm{u_0}^{2p}\right),
\end{aligned}
\end{equation}
where $C_T$ is independent of $\var$ and
the choice of solution.

If \ref{item:H-a} holds, the same estimate holds for any variational
martingale solution $\bar u(t)$ of \eqref{eq:SPDEtwo} with initial data
$\bar u_0\in L^{2p}(\Omega,\cal F_0,\P;H)$.
\end{lem}

\begin{proof}
We only prove the estimate for $u_\var(t)$. The proof for $\bar u(t)$ is identical.

Applying the variational It\^o formula to $\norm{u_\var(t)}^{2p}$,
we obtain, for $0<t\le T$ and $p\ge1$,
\begin{align}
\label{HitoNS}
\norm{u_\var(t)}^{2p}
&=
\norm{u_0}^{2p}
-2p\nu\int_0^t
\norm{u_\var(s)}^{2p-2}
\inpro{Au_\var(s),u_\var(s)}\,ds
\notag\\
&\quad
-2p\int_0^t
\norm{u_\var(s)}^{2p-2}
\inpro{B(u_\var(s),u_\var(s)),u_\var(s)}\,ds
\notag\\
&\quad
+2p\int_0^t
\norm{u_\var(s)}^{2p-2}
\inpro{f\left(\frac{s}{\var},u_\var(s),\Law{u_\var(s)}\right),u_\var(s)}\,ds
\notag\\
&\quad
+p\int_0^t
\norm{u_\var(s)}^{2p-2}
\norm{g\left(\frac{s}{\var},u_\var(s),\Law{u_\var(s)}\right)}_{L_2(U,H)}^2\,ds
\notag\\
&\quad
+2p(p-1)\int_0^t
\norm{u_\var(s)}^{2p-4}
\norm{g\left(\frac{s}{\var},u_\var(s),\Law{u_\var(s)}\right)^*u_\var(s)}_U^2\,ds
+M_t,
\end{align}
where
\[
M_t=
2p\int_0^t
\norm{u_\var(s)}^{2p-2}
\inpro{g\left(\frac{s}{\var},u_\var(s),\Law{u_\var(s)}\right)dW_s,u_\var(s)},
\]
and $M_t$ is a continuous local martingale.
It follows from \eqref{gcon} that
\begin{align}\label{31:3}
\norm{u_\var(s)}^{2p-4}\norm{g\left(\frac{s}{\var},u_\var(s),\Law{u_\var(s)}\right)^*u_\var(s)}_U^2
&\le\norm{g\left(\frac{s}{\var},u_\var(s),\Law{u_\var(s)}\right)}_{L_2(U,H)}^2\norm{u_\var(s)}^{2p-2}\notag\\
&\le C\norm{u_\var(s)}^{2p-2}(1+\norm{u_\var(s)}^2+\E\norm{u_\var(s)}^2).
\end{align}
Then by \eqref{BzeroH}, \eqref{fcon}, \eqref{HitoNS} and \eqref{31:3}, we have
\begin{align}
\label{HitoNSineq}
\norm{u_\var(t)}^{2p}
&+
2p\nu\int_0^t
\norm{u_\var(s)}^{2p-2}\norm{u_\var(s)}_V^2\,ds
\notag\\
&\le
\norm{u_0}^{2p}
+
C\int_0^t
\norm{u_\var(s)}^{2p-2}\left(
1+\norm{u_\var(s)}^{2}
+\E\norm{u_\var(s)}^{2}
\right)ds
+M_t.
\end{align}
In view of the Burkholder–Davis–Gundy inequality and Young's inequality, we have
\begin{align}
\label{HmartNS}
\E{\sup\limits_{t\in [0,T]}|M_t|}
&\le C\E\left(\int_{0}^{T} \norm{u_\var(s)}^{4p-4}\norm{g\left(\frac{s}{\var},u_\var(s),\Law{u_\var(s)}\right)^*u_\var(s)}_U^2\,ds \right)^{\frac{1}{2}} \notag \\
&\le C\E\left(\int_{0}^{T}\norm{u_\var(s)}^{4p-2}(1+\norm{u_\var(s)}^2+\E\norm{u_\var(s)}^2)\,ds\right)^{\frac{1}{2}}\notag\\
&\le C\E\left[\sup_{0\le t\le T}\norm{u_\var(t)}^{2p-1}\left(\int_{0}^{T}(1+\norm{u_\var(s)}^2+\E\norm{u_\var(s)}^2)\,ds\right)^\frac{1}{2}\right]\notag\\
&\le \frac{1}{2}\E\sup_{0\le t\le T}\norm{u_\var(t)}^{2p}+C\E\left(\int_{0}^{T}(1+\norm{u_\var(s)}^2+\E\norm{u_\var(s)}^2)\,ds\right)^p\notag\\
&\le \frac{1}{2}\E\sup_{0\le t\le T}\norm{u_\var(t)}^{2p}+C_{T,p}\int_{0}^{T}\left[1+\E\left(\sup_{0\le r\le s}\norm{u_\var(r)}^{2p}\right)\right]\,ds.
\end{align}
For \(N\in\mathbb N\), define
\[
\tau_N
:=
\inf\left\{t\in[0,T]:\norm{u_\var(t)}\ge N\right\}\wedge T.
\]
Applying \eqref{HitoNSineq} up to \(\tau_N\) and estimating the stochastic integral as in the derivation of \eqref{HmartNS}, H\"older's and Young's inequalities give
\[
\begin{aligned}
\E\sup_{0\le t\le T}\norm{u_\var(t\wedge\tau_N)}^{2p}
\le C_{T,p}\Bigg[
1+\E\norm{u_0}^{2p}
&+\int_0^T
\E\sup_{0\le r\le s}\norm{u_\var(r\wedge\tau_N)}^{2p}\,ds\\
&\quad\qquad+\int_0^T
\left(\E\norm{u_\var(s)}^2\right)^p\,ds
\Bigg].
\end{aligned}
\]
Hence, by Gronwall's inequality, we have
\[
\E\sup_{0\le t\le T}\norm{u_\var(t\wedge\tau_N)}^{2p}
\le
C_{T,p}\left[
1+\E\norm{u_0}^{2p}
+\int_0^T
\left(\E\norm{u_\var(s)}^2\right)^p\,ds
\right].
\]
Since \(u_\var\in L^2(\Omega;C([0,T];H))\), letting \(N\to\infty\) and using Fatou's lemma yield
\[
\E\sup_{0\le t\le T}\norm{u_\var(t)}^{2p}
\le
C_{T,p}\left[
1+\E\norm{u_0}^{2p}
+\int_0^T
\left(\E\norm{u_\var(s)}^2\right)^p\,ds
\right]
<\infty.
\]
By \eqref{HitoNSineq}, \eqref{HmartNS}, Young's inequality, and H\"older's inequality, 
\[
\E\sup_{0\le t\le T}\norm{u_\var(t)}^{2p}
\le
2\E\norm{u_0}^{2p}
+
C_T\int_0^T
\left(
1+\E\sup_{0\le r\le s}\norm{u_\var(r)}^{2p}
\right)ds.
\]
Gronwall's inequality gives
\[
\E\left(\sup_{0\le t\le T}\norm{u_\var(t)}^{2p}\right)\le
C_T\left(1+\E\norm{u_0}^{2p}\right).
\]
Taking $t=T$ in \eqref{HitoNSineq}, we also obtain
\[
\E\int_0^T
\norm{u_\var(s)}^{2p-2}\norm{u_\var(s)}_V^2\,ds
\le
C_T\left(1+\E\norm{u_0}^{2p}\right).
\]
The proof is complete.
\end{proof}

The following estimate is proved by Galerkin approximation in the appendix.
\begin{lem}
\label{lemtwo}
Assume that \ref{item:H-f} and \ref{item:H-g} hold.
Let $u_\var(t)$ be any variational martingale solution of
\eqref{eq:SPDEone}, with initial data $u_0$.
For any $p\ge1$ and $T>0$, if
$u_0\in L^{2p}(\Omega,\cal F_0,\P;V)$, then
\[
\E\sup_{0\le t\le T}\norm{u_\var(t)}_V^{2p}
+
\E\int_0^T
\norm{u_\var(t)}_V^{2p-2}\norm{Au_\var(t)}^2\,dt
\le
C_T\left(1+\E\norm{u_0}_V^{2p}\right),
\]
where $C_T$ is independent of $\var$ and the choice of solution.

If \ref{item:H-a} holds, the same estimate holds for any variational
martingale solution $\bar u(t)$ of \eqref{eq:SPDEtwo} with initial data
$\bar u_0\in L^{2p}(\Omega,\cal F_0,\P;V)$.
\end{lem}

By the Poincar\'e inequality, $\norm{\cdot}_V$ and
$\norm{\cdot}_{H^1}$ are equivalent norms on $V$, and hence
\begin{align}
\label{Hnorm}
\sup_{0\le t\le T}\E \norm{u_\var(t)}^{2p}_{H^1}\le C_T(1+\E{\norm{u_0}_{H^1}^{2p}}).
\end{align}

\begin{lem}
\label{Law-con}
Assume that \ref{item:H-f} and \ref{item:H-g} hold.
Let $u_\var(t)$ be any variational martingale solution of
\eqref{eq:SPDEone} with initial data
$u_0\in L^4(\Omega,\cal F_0,\P;V)$.
Then, for every $T>0$, there exists a constant $C_T>0$,
independent of $\var$ and the choice of solution, such that
for any $0\le t_1\le t_2\le T$,
\begin{align*}
\E\norm{u_\var(t_2)-u_\var(t_1)}^2
\le
C_T(t_2-t_1).
\end{align*}

If \ref{item:H-a} also holds, the same result holds for any
variational martingale solution $\bar u$ of \eqref{eq:SPDEtwo}
with initial data $\bar u_0\in L^4(\Omega,\cal F_0,\P;V)$.
\end{lem}

\begin{proof}
Fix $0\le t_1\le t_2\le T$.
Applying It\^o's formula, we obtain
\begin{align}
\label{33:1}
\E\norm{u_\var(t_2)-u_\var(t_1)}^2
&=
-2\nu\E\int_{t_1}^{t_2}
\inpro{Au_\var(r),u_\var(r)-u_\var(t_1)}\,dr
\notag\\
&\quad
-2\E\int_{t_1}^{t_2}
\inpro{B(u_\var(r),u_\var(r)),u_\var(r)-u_\var(t_1)}\,dr
\notag\\
&\quad
+2\E\int_{t_1}^{t_2}
\inpro{
f\left(\frac{r}{\var},u_\var(r),\Law{u_\var(r)}\right),
u_\var(r)-u_\var(t_1)
}\,dr
\notag\\
&\quad
+\E\int_{t_1}^{t_2}
\norm{
g\left(\frac{r}{\var},u_\var(r),\Law{u_\var(r)}\right)
}_{L_2(U,H)}^2\,dr.
\end{align}
For the Stokes term, using
$-2\inpro{a,a-b}=-\norm{a}^2-\norm{a-b}^2+\norm{b}^2$,
with $a=u_\var(r)$ and $b=u_\var(t_1)$ in the $V$-inner product, we have
\begin{align}\label{33:2}
-2\nu\inpro{Au_\var(r),u_\var(r)-u_\var(t_1)}
\le
\nu\norm{u_\var(t_1)}_V^2.
\end{align}
By \eqref{BzeroH}, \eqref{bskew}, \eqref{bloc}, and the
Poincar\'e inequality, one has
\begin{align}\label{33:3}
\left|-\inpro{B(u_\var(r),u_\var(r)),u_\var(r)-u_\var(t_1)}\right|
&=
\left|\inpro{B(u_\var(r),u_\var(r)),u_\var(t_1)}\right| \notag\\
&=
\left|
b(u_\var(r),u_\var(t_1),u_\var(r))
\right| \notag\\
&\le
C\norm{u_\var(t_1)}_V
\norm{u_\var(r)}
\norm{u_\var(r)}_V  \notag\\
&\le
C\norm{u_\var(t_1)}_V
\norm{u_\var(r)}_V^2 .
\end{align}
By \eqref{33:2}, \eqref{33:3}, the Cauchy--Schwarz inequality and Lemma \ref{lemtwo}, 
\begin{align}\label{33:4}
&-2\nu\E\int_{t_1}^{t_2}
\inpro{Au_\var(r),u_\var(r)-u_\var(t_1)}\,dr
-2\E\int_{t_1}^{t_2}
\inpro{B(u_\var(r),u_\var(r)),u_\var(r)-u_\var(t_1)}\,dr\notag\\
&\qquad\le C_T(t_2-t_1)+
C\int_{t_1}^{t_2}
\E\left(
\norm{u_\var(r)}_V^2
\norm{u_\var(t_1)}_V
\right)dr \notag\\
&\qquad\le C_T(t_2-t_1)+
C\int_{t_1}^{t_2}
\left(\E\norm{u_\var(r)}_V^4\right)^{1/2}
\left(\E\norm{u_\var(t_1)}_V^2\right)^{1/2}
dr \notag\\
&\qquad\le C_T(t_2-t_1).
\end{align}
In view of \eqref{fcon}, \eqref{gcon}, and Young's inequality, we obtain
\begin{align}\label{33:5}
&2\inpro{
f\left(\frac{r}{\var},u_\var(r),\Law{u_\var(r)}\right),
u_\var(r)-u_\var(t_1)
}+\norm{g\left(\frac{r}{\var},u_\var(r),\Law{u_\var(r)}\right)
}_{L_2(U,H)}^2\notag\\
&\qquad\le
\norm{u_\var(r)-u_\var(t_1)}^2
+
C\left(
1+\norm{u_\var(r)}^2+\E\norm{u_\var(r)}^2
\right).
\end{align}
Substituting \eqref{33:4}, \eqref{33:5} into \eqref{33:1} and using Lemma \ref{lemone}, we obtain
\[
\E\norm{u_\var(t_2)-u_\var(t_1)}^2
\le
C_T(t_2-t_1)
+
C\int_{t_1}^{t_2}\E\norm{u_\var(r)-u_\var(t_1)}^2\,dr.
\]
By Gronwall's inequality,
\begin{align}\label{lemconabt}
\E\norm{u_\var(t_2)-u_\var(t_1)}^2\le C_{T}(t_2-t_1), \qquad 0\le t_1\le t_2\le T.
\end{align}

By Remark \ref{rem:averaged-coefficients-NS}(ii), 
the same estimate holds for $\bar u$.
\end{proof}

\subsection{Proof of the weak first Bogolyubov theorem}
\begin{proof}[Proof of Theorem \ref{thmT}]
Fix \(T>0\). By Lemma \ref{lemtwo} with \(p=2\), we have
\begin{equation}
\label{eq:thmT-uniform-strong-estimates-NS}
\begin{aligned}
\sup_{0<\var\le1}
\E\sup_{0\le t\le T}\norm{u_\var(t)}_V^4
+
\sup_{0<\var\le1}
\E\int_0^T
\norm{u_\var(t)}_V^2\norm{Au_\var(t)}^2\,dt\le
C_T\left(1+\E\norm{u_0}_V^4\right).
\end{aligned}
\end{equation}
By the Burkholder--Davis--Gundy inequality,
\ref{item:H-g}(ii), and
\eqref{eq:thmT-uniform-strong-estimates-NS}, for
\(0\le s\le t\le T\),
\[
\begin{aligned}
&\E\norm{
\int_s^t
g\left(
\frac{r}{\var},u_\var(r),\Law{u_\var(r)}
\right)dW_r
}_V^4
\\
\le&
C\E\left(
\int_s^t
\norm{
g\left(
\frac{r}{\var},u_\var(r),\Law{u_\var(r)}
\right)
}_{L_2(U,V)}^2\,dr
\right)^2
\\
\le&
C(t-s)
\int_s^t
\left(
1+\E\norm{u_\var(r)}_V^4
\right)dr
\le
C_T(t-s)^2.
\end{aligned}
\]
Hence, for every \(0<\gamma<1/4\), Kolmogorov's continuity criterion
gives
\begin{equation}
\label{eq:thmT-stochastic-compactness-NS}
\sup_{0<\var\le1}
\E
\norm{
\int_0^\cdot
g\left(
\frac{r}{\var},u_\var(r),\Law{u_\var(r)}
\right)dW_r
}_{C^\gamma([0,T];V)}^4
<\infty.
\end{equation}
By \eqref{eq:SPDEone},
\[
\frac{d}{dt}
\left[
u_\var(t)
-
\int_0^t
g\left(
\frac{r}{\var},u_\var(r),\Law{u_\var(r)}
\right)dW_r
\right]
=
-\nu Au_\var(t)
-
B(u_\var(t),u_\var(t))
+
f\left(
\frac{t}{\var},u_\var(t),\Law{u_\var(t)}
\right)
\]
in \(V^*\). Since for $u\in V$,
$\norm{Au}_{V^*}=\norm{u}_V$, $\norm{B(u,u)}_{V^*}\le C\norm{u}_V^2$.
Then, \eqref{fcon},
\eqref{eq:thmT-uniform-strong-estimates-NS}, and Lemma
\ref{lemone} yield
\begin{equation}
\label{eq:thmT-deterministic-compactness-NS}
\sup_{0<\var\le1}
\E
\norm{
u_\var(\cdot)
-
\int_0^\cdot
g\left(
\frac{r}{\var},u_\var(r),\Law{u_\var(r)}
\right)dW_r
}_{W^{1,2}(0,T;V^*)}^2
<\infty.
\end{equation}
By \eqref{eq:thmT-uniform-strong-estimates-NS},
\eqref{eq:thmT-stochastic-compactness-NS}, and
\eqref{eq:thmT-deterministic-compactness-NS},
\begin{equation}
\label{eq:thmT-V-compact-containment-NS}
\begin{aligned}
\lim_{R\to\infty}
\sup_{0<\var\le1}
\P\Bigg(
&
\sup_{0\le t\le T}
\norm{
u_\var(t)
-
\int_0^t
g\left(
\frac{r}{\var},u_\var(r),\Law{u_\var(r)}
\right)dW_r
}_V
\\
&+
\norm{
u_\var(\cdot)
-
\int_0^\cdot
g\left(
\frac{r}{\var},u_\var(r),\Law{u_\var(r)}
\right)dW_r
}_{W^{1,2}(0,T;V^*)}
\\
&+
\norm{
\int_0^\cdot
g\left(
\frac{r}{\var},u_\var(r),\Law{u_\var(r)}
\right)dW_r
}_{C^\gamma([0,T];V)}
>R
\Bigg)
=0.
\end{aligned}
\end{equation}
By the interpolation inequality between \(V\), \(H\), and \(V^*\),
and the Cauchy--Schwarz inequality,
\[
\begin{aligned}
&\norm{
\left(
u_\var(t)
-
\int_0^t
g\left(
\frac{r}{\var},u_\var(r),\Law{u_\var(r)}
\right)dW_r
\right)
-
\left(
u_\var(s)
-
\int_0^s
g\left(
\frac{r}{\var},u_\var(r),\Law{u_\var(r)}
\right)dW_r
\right)
}
\\
&\qquad\le
C
\norm{
u_\var(\cdot)
-
\int_0^\cdot
g\left(
\frac{r}{\var},u_\var(r),\Law{u_\var(r)}
\right)dW_r
}_{L^\infty(0,T;V)}^{1/2}
\\
&\qquad\quad\times
\norm{
u_\var(\cdot)
-
\int_0^\cdot
g\left(
\frac{r}{\var},u_\var(r),\Law{u_\var(r)}
\right)dW_r
}_{W^{1,2}(0,T;V^*)}^{1/2}
|t-s|^{1/4}.
\end{aligned}
\]
Therefore, by the compact embedding \(V\hookrightarrow H\),
the Arzelà--Ascoli theorem and
\eqref{eq:thmT-V-compact-containment-NS} give
\begin{equation}
\label{eq:thmT-path-tightness-NS}
\left\{
\Law{u_\var(\cdot)}:0<\var\le1
\right\}
\quad\text{is tight in }
\cal P(C([0,T];H)).
\end{equation}
Together with Prokhorov's theorem,
\begin{equation}
\label{eq:thmT-relatively-compact-NS}
\left\{
\Law{u_\var(\cdot)}:0<\var\le1
\right\}
\quad\text{is relatively compact in }
\cal P(C([0,T];H)).
\end{equation}

Fix \(0<\delta<1\). For
\(0\le k\le \lfloor T/\delta\rfloor-1\) and
\(s\in[k\delta,(k+1)\delta)\), we decompose
\begin{align}\label{93:1}
f\left(
\frac{s}{\var},u_\var(s),\Law{u_\var(s)}
\right)
-
\bar f\left(
u_\var(s),\Law{u_\var(s)}
\right)
&=
f\left(
\frac{s}{\var},u_\var(s),\Law{u_\var(s)}
\right)
-
f\left(
\frac{s}{\var},u_\var(k\delta),\Law{u_\var(k\delta)}
\right)
\notag\\
&\quad+
f\left(
\frac{s}{\var},u_\var(k\delta),\Law{u_\var(k\delta)}
\right)
-
\bar f\left(
u_\var(k\delta),\Law{u_\var(k\delta)}
\right)\notag\\
&\quad+
\bar f\left(
u_\var(k\delta),\Law{u_\var(k\delta)}
\right)
-
\bar f\left(
u_\var(s),\Law{u_\var(s)}
\right)\notag\\
&=:A_1(s)+A_2(s)+A_3(s).
\end{align}
It follows from \ref{item:H-f}, Remark
\ref{rem:averaged-coefficients-NS}(ii), and \eqref{lemconabt} that
\begin{align}\label{93:2}
\E\left\|
A_1(s)
\right\|^2+
\E\left\|
A_3(s)
\right\|^2
\le
C_T(s-k\delta).
\end{align}
By \eqref{wf}, we have
\begin{align}\label{93:3}
\left\|
\int_{k\delta}^{(k+1)\delta}
\left[
A_2(s)
\right]ds
\right\|
\le
\delta\,
\omega^f\left(\frac{\delta}{\var}\right)
\left(
1+\norm{u_\var(k\delta)}
+
\left(\E\norm{u_\var(k\delta)}^2\right)^{1/2}
\right).
\end{align}
By \eqref{fcon}, Remark \ref{rem:averaged-coefficients-NS}(ii),
and Lemma \ref{lemone},
\begin{align}
\label{eq:thmT-drift-remainder-NS}
&\E\sup_{0\le t\le T}
\left\|
\int_{\lfloor t/\delta\rfloor\delta}^{t}
\left[
f\left(\frac{s}{\var},u_\var(s),\Law{u_\var(s)}\right)
-\bar f\left(u_\var(s),\Law{u_\var(s)}\right)
\right]ds
\right\|^2\notag\\
&\qquad\le
C\delta^2\left(
1+\E\sup_{0\le s\le T}\norm{u_\var(s)}^2
\right)
\le C_T\delta^2.
\end{align}
Combining \eqref{93:1}, \eqref{93:2}, \eqref{93:3}, and
\eqref{eq:thmT-drift-remainder-NS}, we obtain
\begin{equation}
\label{eq:thmT-drift-residual-NS}
\begin{aligned}
&\E\sup_{0\le t\le T}
\left\|
\int_0^t
\left[
f\left(
\frac{s}{\var},u_\var(s),\Law{u_\var(s)}
\right)
-
\bar f\left(
u_\var(s),\Law{u_\var(s)}
\right)
\right]ds
\right\|^2\le
C_T
\left[
\delta+
\left(
\omega^f\left(\frac{\delta}{\var}\right)
\right)^2
\right].
\end{aligned}
\end{equation}
Similarly, for \(s\in[k\delta,(k+1)\delta)\), we write
\[
\begin{aligned}
g\left(
\frac{s}{\var},u_\var(s),\Law{u_\var(s)}
\right)
-
\bar g\left(
u_\var(s),\Law{u_\var(s)}
\right)
&=
g\left(
\frac{s}{\var},u_\var(s),\Law{u_\var(s)}
\right)
-
g\left(
\frac{s}{\var},u_\var(k\delta),\Law{u_\var(k\delta)}
\right)
\\
&\quad+
g\left(
\frac{s}{\var},u_\var(k\delta),\Law{u_\var(k\delta)}
\right)
-
\bar g\left(
u_\var(k\delta),\Law{u_\var(k\delta)}
\right)
\\
&\quad+
\bar g\left(
u_\var(k\delta),\Law{u_\var(k\delta)}
\right)
-
\bar g\left(
u_\var(s),\Law{u_\var(s)}
\right).
\end{aligned}
\]
Using \ref{item:H-g}(i), Remark
\ref{rem:averaged-coefficients-NS}(ii), \eqref{lemconabt}, and
\eqref{wg} on $[k\delta,(k+1)\delta]$, and using \eqref{gcon}
and Lemma \ref{lemone} on
$[\lfloor T/\delta\rfloor\delta,T]$, we get
\[
\begin{aligned}
&\E\int_0^T
\left\|
g\left(
\frac{s}{\var},u_\var(s),\Law{u_\var(s)}
\right)
-
\bar g\left(
u_\var(s),\Law{u_\var(s)}
\right)
\right\|_{L_2(U,H)}^2\,ds
\le
C_T
\left[
\delta+
\omega^g\left(\frac{\delta}{\var}\right)
\right].
\end{aligned}
\]
The Burkholder--Davis--Gundy inequality then gives
\begin{equation}
\label{eq:thmT-diffusion-residual-NS}
\begin{aligned}
&\E\sup_{0\le t\le T}
\left\|
\int_0^t
\left[
g\left(
\frac{s}{\var},u_\var(s),\Law{u_\var(s)}
\right)
-
\bar g\left(
u_\var(s),\Law{u_\var(s)}
\right)
\right]dW_s
\right\|^2
\le
C_T
\left[
\delta+
\omega^g\left(\frac{\delta}{\var}\right)
\right].
\end{aligned}
\end{equation}
Choosing \(\delta=\sqrt{\var}\) in
\eqref{eq:thmT-drift-residual-NS} and
\eqref{eq:thmT-diffusion-residual-NS}, we conclude that
\begin{equation}
\label{eq:thmT-residual-limit-NS}
\begin{aligned}
\lim_{\var\to0}
\E\sup_{0\le t\le T}
\Bigg(
&
\left\|
\int_0^t
\left[
f\left(
\frac{s}{\var},u_\var(s),\Law{u_\var(s)}
\right)
-
\bar f\left(
u_\var(s),\Law{u_\var(s)}
\right)
\right]ds
\right\|^2
\\
&+
\left\|
\int_0^t
\left[
g\left(
\frac{s}{\var},u_\var(s),\Law{u_\var(s)}
\right)
-
\bar g\left(
u_\var(s),\Law{u_\var(s)}
\right)
\right]dW_s
\right\|^2
\Bigg)
=0.
\end{aligned}
\end{equation}

Let $\var_n\to0$. By \eqref{eq:thmT-relatively-compact-NS},
the path laws have a convergent subsequence.
Fix any such subsequence, still denoted by $n$.
There exists a $C([0,T];H)$-valued random variable $u^1$ such that
\begin{equation}
\label{93:11}
\Law{u_{\var_n}}
\Rightarrow
\Law{u^1}
\quad\text{in }\cal P(C([0,T];H)).
\end{equation}
Let $U_0$ be a separable Hilbert space such that
$U\hookrightarrow U_0$ is Hilbert--Schmidt.
For each $n$, write $W$ for the Wiener process driving $u_{\var_n}$.
These Wiener processes have the same law on $C([0,T];U_0)$.
Together with \eqref{eq:thmT-path-tightness-NS}, this implies that
\[\{\Law{(u_{\var_n},W)}\}_{n\ge1} \text{is tight on}
C([0,T];H)\times C([0,T];U_0).\]
Hence, passing to a further subsequence, we obtain
\begin{equation}
\label{eq:thmT-joint-law-limit-NS}
\Law{(u_{\var_n},W)}
\Rightarrow
\Law{(u^1,W^*)}.
\end{equation}
Since the evaluation map at $t=0$ is continuous,
$\Law{u^1(0)}=\Law{u_0}$.
By Lemma \ref{lemone},
\[
\sup_{n\ge1}
\E\sup_{0\le t\le T}\norm{u_{\var_n}(t)}^4<\infty,
\]
and hence
$\left\{
\sup_{0\le t\le T}\norm{u_{\var_n}(t)}^2
:n\ge1
\right\}$
is uniformly integrable. Since
$x\mapsto\sup_{0\le t\le T}\norm{x(t)}^2$
is continuous on \(C([0,T];H)\), \eqref{93:11} and uniform
integrability imply
\[
\E\sup_{0\le t\le T}\norm{u_{\var_n}(t)}^2
\longrightarrow
\E\sup_{0\le t\le T}\norm{u^1(t)}^2.
\]
Therefore, by the characterization of convergence in \(W_2\),
\[
W_2\left(
\Law{u_{\var_n}},
\Law{u^1}
\right)
\longrightarrow0
\quad\text{in }\cal P_2(C([0,T];H)).
\]
Since \(x\mapsto x(t)\) is \(1\)-Lipschitz from
\(C([0,T];H)\) to \(H\), for every \(t\in[0,T]\),
\[
W_2\left(
\Law{u_{\var_n}(t)},
\Law{u^1(t)}
\right)
\le
W_2\left(
\Law{u_{\var_n}},
\Law{u^1}
\right).
\]
Thus
\begin{equation}
\label{eq:thmT-limit-law-flow-NS}
\sup_{0\le t\le T}
W_2\left(
\Law{u_{\var_n}(t)},
\Law{u^1(t)}
\right)
\longrightarrow0.
\end{equation}

By \eqref{eq:thmT-joint-law-limit-NS}, weak lower semicontinuity, and
\eqref{eq:SPDEthree},
\begin{align}
\label{eq:thmT-limit-regularity-NS}
&\E\sup_{0\le t\le T}\norm{u^1(t)}^2
+\E\int_0^T\norm{u^1(t)}_V^2\,dt\notag\\
&\quad\le
\liminf_{n\to\infty}
\left[
\E\sup_{0\le t\le T}\norm{u_{\var_n}(t)}^2
+\E\int_0^T\norm{u_{\var_n}(t)}_V^2\,dt
\right]
\le C_T\left(1+\E\norm{u_0}^2\right).
\end{align}
Let $\{e_j\}_{j\ge1}$ be an orthonormal basis of $H$ consisting
of Stokes eigenfunctions. For $j\ge1$, set
\[
\begin{aligned}
M_{\var_n,j}(t)
&:=
\inpro{u_{\var_n}(t)-u_{\var_n}(0),e_j}
+\nu\int_0^t\inpro{u_{\var_n}(s),Ae_j}\,ds\\
&\quad+
\int_0^t b(u_{\var_n}(s),u_{\var_n}(s),e_j)\,ds
-\int_0^t
\inpro{
f\left(\frac{s}{\var_n},u_{\var_n}(s),\Law{u_{\var_n}(s)}\right),
e_j}\,ds .
\end{aligned}
\]
By \eqref{eq:SPDEone}, $M_{\var_n,j}$ is a continuous
square-integrable martingale satisfying
\[
\begin{aligned}
\left\langle M_{\var_n,i},M_{\var_n,j}\right\rangle_t
&=
\int_0^t
\left\langle
g\left(\frac{s}{\var_n},u_{\var_n}(s),\Law{u_{\var_n}(s)}\right)^*e_i,
g\left(\frac{s}{\var_n},u_{\var_n}(s),\Law{u_{\var_n}(s)}\right)^*e_j
\right\rangle_U\,ds,\\
\left\langle M_{\var_n,j},\inpro{W,h}_U\right\rangle_t
&=
\int_0^t
\left\langle
g\left(\frac{s}{\var_n},u_{\var_n}(s),\Law{u_{\var_n}(s)}\right)^*e_j,
h
\right\rangle_U\,ds,
\qquad h\in U.
\end{aligned}
\]
By \eqref{bskew},
\begin{equation}
\label{eq:thmT-nonlinear-continuity-NS}
|b(u,u,e_j)-b(v,v,e_j)|
\le
C_j\left(\norm{u}+\norm{v}\right)\norm{u-v},
\qquad u,v\in V.
\end{equation}
Hence \eqref{eq:thmT-joint-law-limit-NS} and
\eqref{eq:thmT-nonlinear-continuity-NS} give the convergence of the
nonlinear integrals. By \eqref{eq:thmT-limit-law-flow-NS} and Remark
\ref{rem:averaged-coefficients-NS}(ii), the integrals containing
$\bar f$ and $\bar g$ converge to those containing
$u^1$ and $\Law{u^1}$. Moreover,
\eqref{eq:thmT-residual-limit-NS} gives
\[
\E\sup_{0\le t\le T}
\left|
\int_0^t
\inpro{
f\left(\frac{s}{\var_n},u_{\var_n}(s),\Law{u_{\var_n}(s)}\right)
-\bar f\left(u_{\var_n}(s),\Law{u_{\var_n}(s)}\right),
e_j
}\,ds
\right|^2
\longrightarrow0.
\]
By It\^o's isometry and \eqref{eq:thmT-diffusion-residual-NS},
\[
\E\int_0^T
\norm{
g\left(\frac{s}{\var_n},u_{\var_n}(s),\Law{u_{\var_n}(s)}\right)
-\bar g\left(u_{\var_n}(s),\Law{u_{\var_n}(s)}\right)
}_{L_2(U,H)}^2\,ds
\longrightarrow0,
\]
and therefore the corresponding errors in the quadratic and cross
variations tend to zero in $L^1$ by the Cauchy--Schwarz inequality
and \eqref{gcon}.

By the Burkholder--Davis--Gundy inequality, \eqref{gcon}, and
Lemma \ref{lemone},
\[
\sup_{n\ge1}
\E\sup_{0\le t\le T}|M_{\var_n,j}(t)|^4
\le C_{T,j}\left(1+\E\norm{u_0}^4\right),
\qquad
\E\sup_{0\le t\le T}|\inpro{W_t,h}_U|^4
\le C_T\norm{h}_U^4.
\]
Thus the martingales and their products occurring in the identities
above are uniformly integrable.

Let $h\in U$ be such that
$v\mapsto\inpro{v,h}_U$ extends continuously to $U_0$.
Testing
\[
M_{\var_n,j},\qquad
M_{\var_n,i}M_{\var_n,j}
-\left\langle M_{\var_n,i},M_{\var_n,j}\right\rangle,\qquad
M_{\var_n,j}\inpro{W,h}_U
-\left\langle M_{\var_n,j},\inpro{W,h}_U\right\rangle
\]
against bounded continuous functions of finitely many past values
of $(u_{\var_n},W)$, we may pass to the limit by
\eqref{eq:thmT-joint-law-limit-NS} and the estimates above.
Applying the same passage to
$\inpro{W_t,h}_U$ and
$\inpro{W_t,h}_U^2-t\norm{h}_U^2$, and then using density and
L\'evy's characterization, we obtain that $W^*$ is a cylindrical
Wiener process with respect to the usual augmentation of the
filtration generated by $(u^1,W^*)$. Furthermore,
\[
\begin{aligned}
M_j^1(t)
&:=
\inpro{u^1(t)-u^1(0),e_j}
+\nu\int_0^t\inpro{u^1(s),Ae_j}\,ds\\
&\quad+
\int_0^t b(u^1(s),u^1(s),e_j)\,ds
-\int_0^t
\inpro{
\bar f\left(u^1(s),\Law{u^1(s)}\right),e_j
}\,ds
\end{aligned}
\]
is a continuous square-integrable martingale satisfying
\begin{align}
\label{eq:thmT-limit-covariations-NS}
\left\langle M_i^1,M_j^1\right\rangle_t
&=
\int_0^t
\left\langle
\bar g\left(u^1(s),\Law{u^1(s)}\right)^*e_i,
\bar g\left(u^1(s),\Law{u^1(s)}\right)^*e_j
\right\rangle_U\,ds,\notag\\
\left\langle M_j^1,\inpro{W^*,h}_U\right\rangle_t
&=
\int_0^t
\left\langle
\bar g\left(u^1(s),\Law{u^1(s)}\right)^*e_j,h
\right\rangle_U\,ds.
\end{align}
By density, the second identity holds for every $h\in U$.
It follows from \eqref{eq:thmT-limit-covariations-NS} that
\[
M_j^1(t)
=
\int_0^t
\inpro{
\bar g\left(u^1(s),\Law{u^1(s)}\right)dW_s^*,e_j},
\qquad
0\le t\le T,\quad \P\text{-a.s.}
\]
By \eqref{eq:thmT-limit-regularity-NS} and the density of
$\mathrm{span}\{e_j:j\ge1\}$ in $V$, the variational identity holds
for every $\phi\in V$. Hence $u^1$ is a variational martingale
solution of \eqref{eq:SPDEtwo} with initial law $\Law{u_0}$.
Since the convergent subsequence was arbitrary, every limit point
as $\var\to0$ has this property. Together with
\eqref{eq:thmT-relatively-compact-NS}, the proof is complete.
\end{proof}

\begin{rem}
\label{rem:shift-uniform-first}
(i)
For every $T>0$, the family of path laws in Theorem \ref{thmT}
remains relatively compact in $\cal P(C([0,T];H))$ uniformly 
for initial data satisfying a common fourth-moment bound in \(V\).

(ii)
The conclusion of Theorem \ref{thmT} holds on $[s,s+T]$
for every $s\in\R$ and $T>0$.
The path space is $C([s,s+T];H)$, and the initial law is
$\Law{\xi}$, where
$\xi\in L^4(\Omega,\cal F_s,\P;V)$.

(iii)
Assume that \ref{item:H-d} also holds.
For any $\tilde F\in\cal H(F_0)$, the corresponding coefficients
$\tilde f$ and $\tilde g$ satisfy
\ref{item:H-f}, \ref{item:H-g}, \ref{item:H-a}, and
\ref{item:H-d} with the same constants and the same averaged
coefficients $\bar f$ and $\bar g$.
Indeed, these properties hold for every translate
\(\sigma_\tau F_0\), since the bounds in
\ref{item:H-f}, \ref{item:H-g}, and \ref{item:H-d} are uniform in
time, while the supremum over \(a\in\R\) makes
\ref{item:H-a} translation invariant. For
\(\sigma_{\tau_n}F_0\to\tilde F\) in \(d\), passing to the limit on
bounded time intervals preserves
\ref{item:H-f}, \ref{item:H-g}(i), \ref{item:H-a}, and
\ref{item:H-d}; the \(V\)-valued bounds in
\ref{item:H-g}(ii) follow by weak lower semicontinuity in
\(L_2(U,V)\). Thus every element of the hull has the stated uniform
properties.
\end{rem}

\section{The Weak Second Bogolyubov Theorem}

In this section, we prove Theorem \ref{thmR} by constructing bounded complete variational solution laws through pullback limits and using uniform moment estimates, finite-window averaging, and tightness to identify their limit points.

Here we formulate complete solutions at the level of path laws
without requiring uniqueness.

\begin{de}
\label{def:complete-variational-solution-law-NS}
A probability law
$\Law{u_\var(\cdot)}
\in
\cal P\bigl(C_{\mathrm{loc}}(\R;H)\bigr)$
is called a \emph{complete variational solution law} of \eqref{eq:SPDEone} if, for every finite interval \([s,t]\subset\R\), its restriction \(\Law{u_\var|_{[s,t]}}\) is a variational solution law of \eqref{eq:SPDEone} on \([s,t]\).
It is called \emph{bounded} if
\[
\sup_{t\in\R}\E\norm{u_\var(t)}^2<\infty.
\]
The same definition applies to \eqref{eq:SPDEtwo}.
\end{de}

\subsection{Uniform pullback estimates and tightness}

\begin{lem}
\label{lem:energy-dissipativity-NS}
Assume that \ref{item:H-f}, \ref{item:H-g}(i), and \ref{item:H-d} hold. Suppose further that $2\nu\lambda_1>c_1+c_2$. 
Let $u_\var(t)$ be any variational martingale solution of
\eqref{eq:SPDEone} with initial value
$u_0\in L^2(\Omega,\cal F_s,\P;H)$.
Then, for all $t\ge s$,
\begin{align}
\label{eq:energy-dissipativity-NS}
\E\norm{u_\var(t)}^2
\le
e^{-\rho(t-s)}\E\norm{u_0}^2
+
\frac{c_0}{\rho},
\end{align}
where $\rho:=2\nu\lambda_1-c_1-c_2>0$.

If, in addition, \ref{item:H-a} holds, the same estimate holds
for any variational martingale solution of the averaged equation
\eqref{eq:SPDEtwo}.
\end{lem}
\begin{proof}
Applying It\^o's formula to \(\norm{u_\var(t)}^2\), using
\eqref{BzeroH}, we obtain
\[
\begin{aligned}
\frac{d}{dt}\E\norm{u_\var(t)}^2
&+
2\nu\E\norm{u_\var(t)}_V^2  \notag\\
&=
\E\left[
2\inpro{
f\left(\frac{t}{\var},u_\var(t),\Law{u_\var(t)}\right),
u_\var(t)
}
+
\norm{
g\left(\frac{t}{\var},u_\var(t),\Law{u_\var(t)}\right)
}_{L_2(U,H)}^2
\right].
\end{aligned}
\]
By \ref{item:H-d},
\begin{align}\label{37:2}
\frac{d}{dt}\E\norm{u_\var(t)}^2
+
2\nu\E\norm{u_\var(t)}_V^2
\le
c_0+(c_1+c_2)\E\norm{u_\var(t)}^2.
\end{align}
Using \eqref{37:2} and the Poincar\'e inequality, one has
\[
\begin{aligned}
\frac{d}{dt}\E\norm{u_\var(t)}^2
\le
-\rho\E\norm{u_\var(t)}^2+c_0,
\end{aligned}
\]
where
$\rho=2\nu\lambda_1-c_1-c_2>0$.
Gronwall's inequality gives 
\[
\E\norm{u_\var(t)}^2
\le
e^{-\rho(t-s)}\E\norm{u_0}^2
+
\frac{c_0}{\rho}
\left(1-e^{-\rho(t-s)}\right)
\le
e^{-\rho(t-s)}\E\norm{u_0}^2
+
\frac{c_0}{\rho}.
\]
By Remark \ref{rem:bar}, the argument yielding
\eqref{eq:energy-dissipativity-NS} applies to $\bar u$.
\end{proof}

\begin{lem}
\label{lem:finite-window-law-stability-NS}
Assume that \ref{item:H-f} and \ref{item:H-g} hold. Fix $0<\var\le1$ and
$a<b$. Let $\{\mu_k\}_{k\ge1}\subset\cal P_2(V)$, 
$\mu\in\cal P_2(V)$ satisfy
$W_2(\mu_k,\mu)\longrightarrow0$,
and for some $p>1$,
\begin{align}
\label{eq:initial-law-uniform-moments-NS}
\sup_{k\ge1}
\int_H\left(\norm{x}_V^2+\norm{x}^{2p}\right)\,\mu_k(dx)
<\infty.
\end{align}
Let $U_k$ be any variational martingale solutions of
\eqref{eq:SPDEone} on $[a,b]$ with initial laws $\mu_k$.
Then
\begin{align}
\label{eq:finite-window-law-stability-NS}
\{\Law{U_k}:k\ge1\}
\quad\text{is relatively compact in }
\cal P(C([a,b];H)).
\end{align}
Every limit point is the path law of a variational martingale
solution of \eqref{eq:SPDEone} on $[a,b]$ with initial law $\mu$.

If, in addition, \ref{item:H-a} holds, the same conclusion holds
for the averaged equation \eqref{eq:SPDEtwo}.
\end{lem}

\begin{proof}
We only prove the assertion for \eqref{eq:SPDEone}. Since
\(W_2(\mu_k,\mu)\to0\), we also have \(\mu_k\Rightarrow\mu\) in \(H\).
By the lower semicontinuity of
\[
x\mapsto \norm{x}^{2p}+\norm{x}_V^2,
\]
where \(\norm{x}_V=+\infty\) for \(x\notin V\),
\[
\int_H
\left(
\norm{x}^{2p}+\norm{x}_V^2
\right)\mu(dx)
\le
\liminf_{k\to\infty}
\int_H
\left(
\norm{x}^{2p}+\norm{x}_V^2
\right)\mu_k(dx)
<\infty.
\]
By Lemma \ref{lemone}, Lemma \ref{lemtwo} with $p=1$, and
\eqref{eq:initial-law-uniform-moments-NS},
\begin{equation}
\label{eq:finite-window-uniform-estimates-NS}
\sup_{k\ge1}
\left[
\E\sup_{a\le t\le b}\norm{U_k(t)}^{2p}
+
\E\sup_{a\le t\le b}\norm{U_k(t)}_V^2
+
\E\int_a^b\norm{AU_k(t)}^2\,dt
\right]
<\infty.
\end{equation}
For \(a\le s\le t\le b\), the Burkholder--Davis--Gundy inequality,
\ref{item:H-g}(i), and
\eqref{eq:finite-window-uniform-estimates-NS} yield
\[
\begin{aligned}
&\sup_{k\ge1}
\E
\norm{
\int_s^t
g\left(
\frac r\var,U_k(r),\Law{U_k(r)}
\right)dW_r
}^{2p}\\
&\qquad\le
C(t-s)^{p-1}
\sup_{k\ge1}
\int_s^t
\left(
1+\E\norm{U_k(r)}^{2p}
\right)dr
\le
C(t-s)^p.
\end{aligned}
\]
Hence, for every
\(0<\gamma<(p-1)/(2p)\), Kolmogorov's continuity criterion gives
\[
\sup_{k\ge1}
\E
\left\|
\int_a^\cdot
g\left(
\frac r\var,U_k(r),\Law{U_k(r)}
\right)dW_r
\right\|_{C^\gamma([a,b];H)}^{2p}
<\infty.
\]
On the other hand, using
\[
\norm{B(u,u)}
\le
C\norm{u}_V\norm{Au},
\qquad u\in D(A),
\]
the Cauchy--Schwarz inequality, and
\ref{item:H-f}, we have
\[
\begin{aligned}
&\norm{
U_k(t)-U_k(s)
-
\int_s^t
g\left(
\frac r\var,U_k(r),\Law{U_k(r)}
\right)dW_r
}\\
&\quad\le
C(t-s)^{1/2}
\left(
1+\sup_{a\le r\le b}\norm{U_k(r)}_V
\right)
\left(
\int_s^t\norm{AU_k(r)}^2\,dr
\right)^{1/2}\\
&\qquad+
C(t-s)
\left[
1+
\sup_{a\le r\le b}\norm{U_k(r)}
+
\sup_{a\le r\le b}
\left(\E\norm{U_k(r)}^2\right)^{1/2}
\right].
\end{aligned}
\]
Together with \eqref{eq:finite-window-uniform-estimates-NS}, the
compact embedding \(V\hookrightarrow H\), and the
Arzel\`a--Ascoli theorem, one has
\begin{equation}
\label{eq:finite-window-path-law-tightness-NS}
\left\{
\Law{U_k}:k\ge1
\right\}
\quad\text{is tight in }\cal P(C([a,b];H)).
\end{equation}
By \eqref{eq:finite-window-path-law-tightness-NS} and
Prokhorov's theorem, the path laws are relatively compact.
Fix any convergent subsequence, still denoted by $\Law{U_k}$.
There exists a $C([a,b];H)$-valued random variable $U_*$ such that
\begin{align}
\label{94:1}
\Law{U_k}
\Rightarrow
\Law{U_*}
\quad\text{in }\cal P(C([a,b];H)).
\end{align}
For each $k$, write $W$ for the driving Wiener process,
normalized by $W_a=0$.
Let $U_0$ be as in the proof of Theorem \ref{thmT}.
By the argument yielding \eqref{eq:thmT-joint-law-limit-NS},
after passing to a further subsequence, we obtain
\[
\Law{(U_k,W)}
\Rightarrow
\Law{(U_*,W^*)}
\quad\text{in }
\cal P\bigl(C([a,b];H)\times C([a,b];U_0)\bigr).
\]
By \eqref{eq:finite-window-uniform-estimates-NS} and \(p>1\),
\[
\sup_{k\ge1}
\E\left(
\sup_{a\le t\le b}\norm{U_k(t)}^2
\right)^p
=
\sup_{k\ge1}
\E\sup_{a\le t\le b}\norm{U_k(t)}^{2p}
<\infty.
\]
Hence
$\left\{
\sup_{a\le t\le b}\norm{U_k(t)}^2
\right\}_{k\ge1}$
is uniformly integrable. Together with
\eqref{94:1},
we obtain
\begin{equation}
\label{eq:finite-window-path-W2-convergence-NS}
W_2\left(
\Law{U_k},\Law{U_*}
\right)\longrightarrow0
\quad\text{in }\cal P_2(C([a,b];H)).
\end{equation}
Since the evaluation map \(x\mapsto x(t)\) is \(1\)-Lipschitz on
\(C([a,b];H)\), \eqref{eq:finite-window-path-W2-convergence-NS} gives
\begin{equation}
\label{eq:finite-window-marginal-W2-convergence-NS}
\sup_{a\le t\le b}
W_2\left(
\Law{U_k(t)},\Law{U_*(t)}
\right)
\longrightarrow0.
\end{equation}
Moreover, by the weak convergence of the path laws and
\(W_2(\mu_k,\mu)\to0\), we get
$\Law{U_*(a)}=\mu$.

Let \(e_j\in D(A)\) be an eigenfunction of the Stokes operator. For each
\(k\),
\[
\begin{aligned}
\mathcal M_{k,j}(t)
&:=
\inpro{U_k(t)-U_k(a),e_j}
+
\nu\int_a^t\inpro{U_k(r),e_j}_V\,dr
+
\int_a^t b(U_k(r),U_k(r),e_j)\,dr\\
&\quad-
\int_a^t
\inpro{
f\left(
\frac r\var,U_k(r),\Law{U_k(r)}
\right),
e_j
}\,dr
\end{aligned}
\]
is a continuous square-integrable martingale with
\[
\left\langle
\mathcal M_{k,i},\mathcal M_{k,j}
\right\rangle_t
=
\int_a^t
\left\langle
g\left(
\frac r\var,U_k(r),\Law{U_k(r)}
\right)^*e_i,
g\left(
\frac r\var,U_k(r),\Law{U_k(r)}
\right)^*e_j
\right\rangle_U\,dr.
\]
By the Burkholder--Davis--Gundy inequality, \eqref{gcon}, and
\eqref{eq:finite-window-uniform-estimates-NS},
\[
\begin{aligned}
\sup_{k\ge1}\E\sup_{a\le t\le b}
|\mathcal M_{k,j}(t)|^{2p}
&\le
C_{p,j}\sup_{k\ge1}\E\left[
\int_a^b
\norm{
g\left(\frac r\var,U_k(r),\Law{U_k(r)}\right)
}_{L_2(U,H)}^2\,dr
\right]^p
<\infty.
\end{aligned}
\]
The Wiener coordinates also have finite $2p$-moments.
Since $p>1$, the martingales, their products, and their quadratic
and cross variations are uniformly integrable.
By \eqref{eq:finite-window-path-W2-convergence-NS},
\eqref{eq:finite-window-marginal-W2-convergence-NS},
\eqref{eq:thmT-nonlinear-continuity-NS},
\ref{item:H-f}, and \ref{item:H-g}(i), the argument yielding
\eqref{eq:thmT-limit-covariations-NS} gives
\[
\begin{aligned}
\mathcal M_{*,j}(t)
&:=
\inpro{U_*(t)-U_*(a),e_j}
+\nu\int_a^t\inpro{U_*(r),e_j}_V\,dr\\
&\quad+
\int_a^t b(U_*(r),U_*(r),e_j)\,dr
-\int_a^t
\inpro{
f\left(\frac r\var,U_*(r),\Law{U_*(r)}\right),e_j
}\,dr
\end{aligned}
\]
as a continuous square-integrable martingale satisfying
\[
\begin{aligned}
\left\langle\mathcal M_{*,i},\mathcal M_{*,j}\right\rangle_t
&=
\int_a^t
\left\langle
g\left(\frac r\var,U_*(r),\Law{U_*(r)}\right)^*e_i,
g\left(\frac r\var,U_*(r),\Law{U_*(r)}\right)^*e_j
\right\rangle_U\,dr,\\
\left\langle\mathcal M_{*,j},\inpro{W^*,h}_U\right\rangle_t
&=
\int_a^t
\left\langle
g\left(\frac r\var,U_*(r),\Law{U_*(r)}\right)^*e_j,h
\right\rangle_U\,dr,
\qquad h\in U.
\end{aligned}
\]
Here $W^*$ is a cylindrical Wiener process with respect to the
usual augmentation of the filtration generated by $(U_*,W^*)$.
By the argument following \eqref{eq:thmT-limit-covariations-NS},
\[
\mathcal M_{*,j}(t)
=
\int_a^t
\inpro{
g\left(\frac r\var,U_*(r),\Law{U_*(r)}\right)dW_r^*,e_j
},
\qquad a\le t\le b,\quad\P\text{-a.s.}
\]
By \eqref{eq:finite-window-uniform-estimates-NS} and
weak lower semicontinuity,
\[
\E\sup_{a\le t\le b}\norm{U_*(t)}^2
+
\E\int_a^b\norm{U_*(t)}_V^2\,dt
<\infty.
\]
Together with $\Law{U_*(a)}=\mu$ and the density of the Stokes
eigenfunctions in $V$, this shows that $U_*$ is a variational
martingale solution of \eqref{eq:SPDEone} with initial law $\mu$.
Since the convergent subsequence was arbitrary, every limit point
has this property.

Under \ref{item:H-a}, Remark
\ref{rem:averaged-coefficients-NS}(ii) gives the same coefficient
bounds for $\bar f,\bar g$.
The argument yielding
\eqref{eq:finite-window-law-stability-NS} and identifying its
limit points therefore applies to \eqref{eq:SPDEtwo}.
\end{proof}

\begin{lem}
\label{lem:complete-law-NS}
Assume that \ref{item:H-f}, \ref{item:H-g}, \ref{item:H-a} and
\ref{item:H-d} hold with
\(2\nu\lambda_1>c_1+c_2\).
Then, for every \(0<\var\le1\), equation \eqref{eq:SPDEone} admits a
bounded complete variational solution law. Moreover, for some
\(p>1\), these solution laws can be chosen such that
\begin{equation}
\label{eq:complete-uniform-bound-NS}
\sup_{0<\var\le1}\sup_{t\in\R}
\E\left(
\norm{u_\var(t)}^2
+
\norm{u_\var(t)}_V^2
+
\norm{u_\var(t)}^{2p}
\right)
<\infty.
\end{equation}
The averaged equation \eqref{eq:SPDEtwo} admits a bounded complete
variational solution law satisfying the same estimate.
\end{lem}

\begin{proof}
For every $0<\var\le1$ and $n\in\N$, let $u_\var^n$ be a
variational martingale solution of \eqref{eq:SPDEone} on
$[-n,\infty)$ with $u_\var^n(-n)=0$.
By Lemma \ref{lem:energy-dissipativity-NS},
\begin{equation}
\label{eq:pullback-H-bound-NS}
\sup_{0<\var\le1}\sup_{n\in\mathbb N}\sup_{t\ge-n}
\E\norm{u_\var^n(t)}^2
\le
\frac{c_0}{2\nu\lambda_1-c_1-c_2}.
\end{equation}
For \(t\ge-n+1\), using \eqref{37:2} and \eqref{eq:pullback-H-bound-NS}, we obtain
\begin{equation}
\label{eq:pullback-integral-V-bound-NS}
\sup_{0<\var\le1}\sup_{n\in\mathbb N}\sup_{t\ge-n+1}
\int_{t-1}^t\E\norm{u_\var^n(r)}_V^2\,dr
<\infty.
\end{equation}
Hence, for every \(t\ge-n+1\), there exists \(s\in[t-1,t]\) such that
\[
\E\norm{u_\var^n(s)}_V^2
\le
\int_{t-1}^t\E\norm{u_\var^n(r)}_V^2\,dr
\le C.
\]
Applying Lemma \ref{lemtwo} on \([s,s+1]\), we have
\begin{equation}
\label{eq:pullback-V-bound-NS}
\sup_{0<\var\le1}\sup_{n\in\mathbb N}\sup_{t\ge-n+1}
\E\norm{u_\var^n(t)}_V^2
<\infty.
\end{equation}
For $p>1$, applying It\^o's formula to
$\norm{u_\var^n(t)}^{2p}$ and using \eqref{BzeroH},
\eqref{Hfour}, and \eqref{gcon}, we obtain
\[
\begin{aligned}
\frac{d}{dt}\E\norm{u_\var^n(t)}^{2p}
&+
2p\nu
\E\left(
\norm{u_\var^n(t)}^{2p-2}
\norm{u_\var^n(t)}_V^2
\right)\\
&\le
pc_0\E\norm{u_\var^n(t)}^{2p-2}
+
pc_1\E\norm{u_\var^n(t)}^{2p}\\
&\quad+
pc_2
\E\norm{u_\var^n(t)}^2
\E\norm{u_\var^n(t)}^{2p-2}\\
&\quad+
2p(p-1)C
\left[
\E\norm{u_\var^n(t)}^{2p-2}
+
\E\norm{u_\var^n(t)}^{2p}\right.\\
&\hspace{3.5cm}\left.
+
\E\norm{u_\var^n(t)}^2
\E\norm{u_\var^n(t)}^{2p-2}
\right].
\end{aligned}
\]
By H\"older's inequality,
\[
\E\norm{u_\var^n(t)}^2
\E\norm{u_\var^n(t)}^{2p-2}
\le
\left(
\E\norm{u_\var^n(t)}^{2p}
\right)^{1/p}
\left(
\E\norm{u_\var^n(t)}^{2p}
\right)^{(p-1)/p}
=
\E\norm{u_\var^n(t)}^{2p}.
\]
Thus,
\[
\begin{aligned}
\frac{d}{dt}\E\norm{u_\var^n(t)}^{2p}
&+
2p\nu
\E\left(
\norm{u_\var^n(t)}^{2p-2}
\norm{u_\var^n(t)}_V^2
\right)\\
&\le
p\left[
c_1+c_2+4(p-1)C
\right]
\E\norm{u_\var^n(t)}^{2p}\\
&\quad+
p\left[
c_0+2(p-1)C
\right]
\E\norm{u_\var^n(t)}^{2p-2}.
\end{aligned}
\]
Young's inequality and the Poincar\'e inequality imply for every $\eta>0$,
\[
\begin{aligned}
&2p\nu
\E\left[
\norm{u_\var^n(t)}^{2p-2}
\norm{u_\var^n(t)}_V^2
\right]
-
p\left[
c_1+c_2+4(p-1)C
\right]
\E\norm{u_\var^n(t)}^{2p}\\
&\qquad
-
p\left[
c_0+2(p-1)C
\right]
\E\norm{u_\var^n(t)}^{2p-2}\\
&\ge
p\left[
2\nu\lambda_1-c_1-c_2-4(p-1)C
\right]
\E\norm{u_\var^n(t)}^{2p}\\
&\qquad
-
p\left[
c_0+2(p-1)C
\right]
\left(
\eta\E\norm{u_\var^n(t)}^{2p}
+
C_{p,\eta}
\right)\\
&=
p\Big[
2\nu\lambda_1-c_1-c_2
-4(p-1)C
-\eta\big(
c_0+2(p-1)C
\big)
\Big]
\E\norm{u_\var^n(t)}^{2p}
-
C_{p,\eta}.
\end{aligned}
\]
Then we have
\begin{equation}
\label{eq:pullback-higher-moment-differential-NS}
\begin{aligned}
\frac{d}{dt}\E\norm{u_\var^n(t)}^{2p}
&+
p\Big[
2\nu\lambda_1-c_1-c_2
-4(p-1)C\\
&\qquad
-\eta\big(
c_0+2(p-1)C
\big)
\Big]
\E\norm{u_\var^n(t)}^{2p}
\le
C_{p,\eta}.
\end{aligned}
\end{equation}
Since \(2\nu\lambda_1>c_1+c_2\), we can choose \(p>1\) sufficiently close
to \(1\) and then \(\eta>0\) sufficiently small such that
\[
2\nu\lambda_1-c_1-c_2
-4(p-1)C
-\eta\bigl(c_0+2(p-1)C\bigr)>0.
\]
Then \eqref{eq:pullback-higher-moment-differential-NS} and
\(u_\var^n(-n)=0\) imply
\begin{equation}
\label{eq:pullback-higher-moment-NS}
\sup_{0<\var\le1}
\sup_{n\in\mathbb N}
\sup_{t\ge-n}
\E\norm{u_\var^n(t)}^{2p}
<\infty.
\end{equation}

Fix \(0<\var\le1\) and \(R\in\mathbb N\). By
\eqref{eq:pullback-V-bound-NS}, the compact embedding
\(V\hookrightarrow H\), and Chebyshev's inequality,
$\left\{
\Law{u_\var^n(-R)}:n\ge R+1
\right\}$
is tight in \(\cal P(H)\). Moreover, by
\eqref{eq:pullback-higher-moment-NS},
\[
\sup_{n\ge R+1}
\int_{\{\norm{x}>M\}}
\norm{x}^2\,\Law{u_\var^n(-R)}(dx)
\le
CM^{-2(p-1)}
\longrightarrow0,
\qquad M\to\infty.
\]
By Prokhorov's theorem and
\cite[Theorem 6.9]{villani2009optimal}, this family is relatively
compact in \((\cal P_2(H),W_2)\).
A diagonal argument gives a sequence
\(n_j\to\infty\) and \(\mu_\var^{-R}\in\cal P_2(H)\),
\(R\in\mathbb N\), such that
\begin{equation}
\label{eq:initial-W2-limit-NS}
W_2\left(
\Law{u_\var^{n_j}(-R)},\mu_\var^{-R}
\right)
\longrightarrow0,
\qquad j\to\infty,
\end{equation}
for every fixed \(R\). By lower semicontinuity,
\[
\sup_{R\in\mathbb N}
\int_H
\left(
\norm{x}_V^2+\norm{x}^{2p}
\right)\mu_\var^{-R}(dx)
<\infty.
\]
By Lemma \ref{lem:finite-window-law-stability-NS},
\eqref{eq:initial-W2-limit-NS}, \eqref{eq:pullback-V-bound-NS},
and \eqref{eq:pullback-higher-moment-NS}, the laws
$\Law{u_\var^{n_j}|_{[-R,R]}}$ are relatively compact for every
$R\in\N$.
A further diagonal extraction gives, along the same subsequence
$\{n_j\}$,
\begin{equation}
\label{eq:window-path-law-limit-NS}
\Law{u_\var^{n_j}|_{[-R,R]}}
\Rightarrow
\Law{u_{\var,R}}
\quad\text{in }\cal P(C([-R,R];H)),
\qquad R\in\N,
\end{equation}
where $u_{\var,R}$ is a variational martingale solution of
\eqref{eq:SPDEone} on $[-R,R]$ with initial law
$\mu_\var^{-R}$ at time $-R$.
For \(R_2>R_1\), using the continuity of the restriction map and
\eqref{eq:window-path-law-limit-NS}, we obtain
\begin{equation}
\label{eq:consistent-window-laws-NS}
\Law{
u_{\var,R_2}|_{[-R_1,R_1]}
}
=
\Law{u_{\var,R_1}}.
\end{equation}
Thus the consistent laws
\(\{\Law{u_{\var,R}}\}_{R\in\mathbb N}\) determine a law on
\(C_{\mathrm{loc}}(\R;H)\). Denote its coordinate process by
\(u_\var\). By \eqref{eq:consistent-window-laws-NS}, for every
\(R\in\mathbb N\),
\[\Law{u_\var|_{[-R,R]}}
=
\Law{u_{\var,R}}.\]
Let \([s,t]\subset\R\) be an arbitrary finite interval and choose
\(R\in\mathbb N\) sufficiently large such that
$[s,t]\subset[-R,R]$.
Then
\[\Law{u_\var|_{[s,t]}}
=
\Law{u_{\var,R}|_{[s,t]}}.\]
Since $u_{\var,R}$ is a variational martingale solution of
\eqref{eq:SPDEone} on \([-R,R]\), its restriction to \([s,t]\) is a
variational martingale solution of \eqref{eq:SPDEone} on $[s,t]$ with initial
law $\Law{u_{\var,R}(s)}=
\Law{u_\var(s)}$.
Therefore,
\(\Law{u_\var(\cdot)}\) is a complete variational solution law of
\eqref{eq:SPDEone}.

Fix \(t\in\R\) and choose \(R\in\mathbb N\) with \(R>|t|\). Using
\eqref{eq:window-path-law-limit-NS}, the Portmanteau theorem, and the
lower semicontinuity on \(C([-R,R];H)\) of
\[
x\mapsto
\norm{x(t)}^2+\norm{x(t)}_V^2+\norm{x(t)}^{2p},
\]
we obtain
\[
\begin{aligned}
&\E\norm{u_\var(t)}^2
+
\E\norm{u_\var(t)}_V^2
+
\E\norm{u_\var(t)}^{2p}\\
&\qquad\le
\liminf_{j\to\infty}
\left[
\E\norm{u_\var^{n_j}(t)}^2
+
\E\norm{u_\var^{n_j}(t)}_V^2
+
\E\norm{u_\var^{n_j}(t)}^{2p}
\right].
\end{aligned}
\]
Using \eqref{eq:pullback-H-bound-NS},
\eqref{eq:pullback-V-bound-NS}, and
\eqref{eq:pullback-higher-moment-NS}, we further obtain
\[
\sup_{0<\var\le1}\sup_{t\in\R}
\E\left(
\norm{u_\var(t)}^2
+
\norm{u_\var(t)}_V^2
+
\norm{u_\var(t)}^{2p}
\right)
<\infty.
\]

By Remarks \ref{rem:averaged-coefficients-NS}(ii) and
\ref{rem:bar}, the argument yielding
\eqref{eq:complete-uniform-bound-NS} applies to
\eqref{eq:SPDEtwo} and gives a bounded complete variational
solution law satisfying
\[
\sup_{t\in\R}
\E\left(
\norm{\bar u(t)}^2
+
\norm{\bar u(t)}_V^2
+
\norm{\bar u(t)}^{2p}
\right)
<\infty.
\]
\end{proof}

\begin{cor}
\label{cor:finite-time-compact-law-averaging-NS}
Assume that \ref{item:H-f}, \ref{item:H-g}, and \ref{item:H-a} hold.
Let $T>0$, and let $K\subset\cal P_2(H)$ be compact in
$(\cal P_2(H),W_2)$.
Assume that, for some $p>1$,
\[
\sup_{\mu\in K}
\int_H\left(\norm{x}_V^2+\norm{x}^{2p}\right)\mu(dx)<\infty.
\]
Let $\var_n\to0$, $\tau_n\in\R$, and $\mu_n\in K$.
For each $n$, let $u_{\var_n}$ be any variational martingale
solution of \eqref{eq:SPDEone} on $[\tau_n,\tau_n+T]$
with initial law $\mu_n$.
Then $\{\Law{u_{\var_n}(\tau_n+\cdot)}:n\ge1\}$ is relatively
compact in $\cal P(C([0,T];H))$.
Every limit point is the path law of a variational martingale
solution of \eqref{eq:SPDEtwo} on $[0,T]$.
Its initial law is the $W_2$-limit of $\mu_n$ along the same
subsequence.
\end{cor}

\begin{proof}
For $0\le t\le T$, set $U_n(t):=u_{\var_n}(\tau_n+t)$.
By Lemma \ref{lemone}, Lemma \ref{lemtwo} with $p=1$, and
the argument yielding
\eqref{eq:finite-window-uniform-estimates-NS}, we obtain
\begin{equation}
\label{eq:compact-averaging-uniform-estimates-NS}
\sup_{n\ge1}
\left[
\E\sup_{0\le t\le T}\norm{U_n(t)}^{2p}
+
\E\sup_{0\le t\le T}\norm{U_n(t)}_V^2
+
\E\int_0^T\norm{AU_n(t)}^2\,dt
\right]
<\infty.
\end{equation}

Repeating the stochastic and deterministic increment estimates used
in the proof of \eqref{eq:finite-window-path-law-tightness-NS} on \([\tau_n,\tau_n+T]\), and using
\eqref{eq:compact-averaging-uniform-estimates-NS}, we can conclude that
\[\{\Law{U_n}:n\ge1\} \text{ is tight in } \cal P(C([0,T];H)).\]
Since tightness gives uniform convergence in
probability of the path moduli of continuity, while
\eqref{eq:compact-averaging-uniform-estimates-NS} and \(p>1\) give the
required uniform integrability, we further obtain
\begin{equation}
\label{eq:compact-averaging-modulus-NS}
\lim_{\delta\downarrow0}
\sup_{n\ge1}
\E
\sup_{\substack{0\le s,t\le T\\|t-s|\le\delta}}
\norm{U_n(t)-U_n(s)}^2
=0.
\end{equation}
By the same partition estimates used to derive
\eqref{eq:thmT-residual-limit-NS}, applied on
\([\tau_n,\tau_n+T]\), and using \ref{item:H-a},
\eqref{eq:compact-averaging-uniform-estimates-NS}, and
\eqref{eq:compact-averaging-modulus-NS}, we obtain
\[
\begin{aligned}
&\E\sup_{0\le t\le T}
\left\|
\int_0^t
\left[
f\left(
\frac{\tau_n+s}{\var_n},
U_n(s),
\Law{U_n(s)}
\right)
-
\bar f\left(
U_n(s),
\Law{U_n(s)}
\right)
\right]ds
\right\|^2\\
&\qquad\le
C_T\left[
\E\sup_{\substack{0\le r,q\le T\\
|r-q|\le\sqrt{\var_n}}}
\norm{U_n(r)-U_n(q)}^2
+
\sqrt{\var_n}
+
\left(
\omega^f\left(\frac{1}{\sqrt{\var_n}}\right)
\right)^2
\right]
\longrightarrow0,
\qquad n\to\infty.
\end{aligned}
\]
Similarly, by \eqref{wg},
\[
\begin{aligned}
&\E\int_0^T
\left\|
g\left(
\frac{\tau_n+s}{\var_n},
U_n(s),
\Law{U_n(s)}
\right)
-
\bar g\left(
U_n(s),
\Law{U_n(s)}
\right)
\right\|_{L_2(U,H)}^2\,ds\\
&\qquad\le
C_T\left[
\E\sup_{\substack{0\le r,q\le T\\
|r-q|\le\sqrt{\var_n}}}
\norm{U_n(r)-U_n(q)}^2
+
\sqrt{\var_n}
+
\omega^g\left(\frac{1}{\sqrt{\var_n}}\right)
\right]
\longrightarrow0,
\qquad n\to\infty.
\end{aligned}
\]
Hence
\begin{equation}
\label{eq:compact-averaging-residual-limit-NS}
\begin{aligned}
&\E\sup_{0\le t\le T}
\left\|
\int_0^t
\left[
f\left(
\frac{\tau_n+s}{\var_n},
U_n(s),
\Law{U_n(s)}
\right)
-
\bar f\left(
U_n(s),
\Law{U_n(s)}
\right)
\right]ds
\right\|^2\\
&\quad+
\E\int_0^T
\left\|
g\left(
\frac{\tau_n+s}{\var_n},
U_n(s),
\Law{U_n(s)}
\right)
-
\bar g\left(
U_n(s),
\Law{U_n(s)}
\right)
\right\|_{L_2(U,H)}^2\,ds
\longrightarrow0,
\qquad n\to\infty.
\end{aligned}
\end{equation}
By Prokhorov's theorem, 
\[\{\Law{U_n}:n\ge1\} \text{ is relatively compact in } \cal P(C([0,T];H)).\]
Fix any convergent subsequence, still denoted by $\{\Law{U_n}\}$, and a
$C([0,T];H)$-valued random variable $U_*$ such that
$\Law{U_n}\Rightarrow\Law{U_*}$.
By the compactness of $K$ in $(\cal P_2(H),W_2)$ and continuity
of evaluation at zero, we obtain
\[W_2(\mu_n,\mu)\to0,\quad\text{where } \mu=\Law{U_*(0)}\in K.\]
Since \eqref{eq:compact-averaging-uniform-estimates-NS} implies
\[
\sup_{n\ge1}
\E\left(
\sup_{0\le t\le T}\norm{U_n(t)}^2
\right)^p
<\infty,
\]
the squared supremum norms are uniformly integrable. Therefore,
\begin{equation}
\label{eq:compact-averaging-path-W2-limit-NS}
\begin{aligned}
W_2\left(
\Law{U_n},\Law{U_*}
\right)&\longrightarrow0
\quad\text{in }\cal P_2(C([0,T];H)),\\
\sup_{0\le t\le T}
W_2\left(
\Law{U_n(t)},\Law{U_*(t)}
\right)&\longrightarrow0,
\qquad
\Law{U_*(0)}=\mu .
\end{aligned}
\end{equation}
Write \(W\) for the shifted driving Wiener process of \(U_n\), with
\(W_0=0\). By the argument yielding
\eqref{eq:thmT-joint-law-limit-NS}, after passing to a further
subsequence,
\[
\Law{(U_n,W)}
\Rightarrow
\Law{(U_*,W^*)}.
\]

By \ref{item:H-f}, \ref{item:H-g}(i),
\eqref{eq:compact-averaging-uniform-estimates-NS},
\eqref{eq:compact-averaging-residual-limit-NS},
\eqref{eq:compact-averaging-path-W2-limit-NS}, and
\eqref{eq:thmT-nonlinear-continuity-NS}, the argument yielding
\eqref{eq:thmT-limit-covariations-NS} applies with \(2p\)-moments
in place of fourth moments,
\eqref{eq:compact-averaging-uniform-estimates-NS} gives the required
uniform integrability. Applying the argument deriving the variational
identity from \eqref{eq:thmT-limit-covariations-NS}, together with
weak lower semicontinuity and
\eqref{eq:compact-averaging-uniform-estimates-NS}, we obtain that
\(U_*\) is a variational martingale solution of
\eqref{eq:SPDEtwo} with initial law \(\mu\). Since the convergent
subsequence was arbitrary, every limit point has this property,
which proves the assertion.
\end{proof}

\subsection{Proof of the weak second Bogolyubov theorem}
With the above estimates at hand, we can now prove the weak second Bogolyubov theorem.

\begin{proof}[Proof of Theorem \ref{thmR}]
Assertions (1) and (2) follow from Lemma
\ref{lem:complete-law-NS}. We prove assertion (3).

Let \(\var_n\to0\), and let
\(\Law{u_{\var_n}(\cdot)}\) be the bounded complete variational
solution law of \eqref{eq:SPDEone} given by
Lemma \ref{lem:complete-law-NS}. For \(R\in\N\), set
\(\mu_n^R:=\Law{u_{\var_n}(-R)}\). By
\eqref{eq:complete-uniform-bound-NS},
\begin{equation}
\label{eq:thmR-initial-law-moment-NS}
\sup_{R\in\N}\sup_{n\ge1}
\int_H
\left(
\norm{x}_V^2+\norm{x}^{2p}
\right)\mu_n^R(dx)
<\infty.
\end{equation}
For each fixed \(R\in\N\), repeating the \(W_2\)-compactness argument
leading to \eqref{eq:initial-W2-limit-NS}, and then using a diagonal
argument in \(R\), we can find a subsequence, still denoted by \(n\),
and \(\mu^{-R}\in\cal P_2(H)\), \(R\in\N\), such that
\begin{equation}
\label{eq:thmR-initial-law-W2-limit-NS}
W_2(\mu_n^R,\mu^{-R})\longrightarrow0,
\qquad R\in\N.
\end{equation}
Moreover, by \eqref{eq:thmR-initial-law-moment-NS} and lower
semicontinuity,
\begin{equation}
\label{eq:thmR-limit-initial-law-moment-NS}
\sup_{R\in\N}
\int_H
\left(
\norm{x}_V^2+\norm{x}^{2p}
\right)\mu^{-R}(dx)
\le C.
\end{equation}
For every fixed \(R\), \eqref{eq:thmR-initial-law-W2-limit-NS} implies
that
\(\{\mu_n^R:n\ge1\}\cup\{\mu^{-R}\}\) is compact in
\((\cal P_2(H),W_2)\), while
\eqref{eq:thmR-initial-law-moment-NS} and
\eqref{eq:thmR-limit-initial-law-moment-NS} give the uniform moment
bound required in Corollary
\ref{cor:finite-time-compact-law-averaging-NS}.

By Definition \ref{def:complete-variational-solution-law-NS},
$\Law{u_{\var_n}|_{[-R,R]}}$ is the path law of a variational
martingale solution of \eqref{eq:SPDEone} on $[-R,R]$
with initial law $\mu_n^R$ at time $-R$.
Applying Corollary \ref{cor:finite-time-compact-law-averaging-NS}
with $\tau_n=-R$ and $T=2R$, 
we obtain 
\[\{\Law{u_{\var_n}|_{[-R,R]}}\} \quad\text{is relatively compact in }
\cal P(C([-R,R];H)).\]
By \eqref{eq:thmR-initial-law-W2-limit-NS}, every limit point
is the path law of a variational martingale solution of
\eqref{eq:SPDEtwo} on $[-R,R]$ with initial law $\mu^{-R}$.
By a diagonal argument and \eqref{eq:consistent-window-laws-NS},
the limiting path laws are consistent.
The construction following \eqref{eq:consistent-window-laws-NS}
then gives a complete variational solution law $\Law{u_*(\cdot)}$
of \eqref{eq:SPDEtwo} such that
\begin{equation}
\label{eq:thmR-window-law-convergence-NS}
\Law{u_{\var_n}|_{[-R,R]}}
\Rightarrow
\Law{u_*|_{[-R,R]}}
\quad\text{in }\cal P(C([-R,R];H)),
\qquad R\in\N.
\end{equation}

Moreover, since \eqref{eq:thmR-window-law-convergence-NS} holds for
every \(R\in\N\) along the same subsequence, we obtain
\begin{equation}
\label{eq:thmR-complete-path-law-convergence-NS}
\Law{u_{\var_n}}
\Rightarrow
\Law{u_*}
\quad\text{in }\cal P(C_{\mathrm{loc}}(\R;H)).
\end{equation}
Fix \(t\in\R\) and choose \(R>|t|\). By
\eqref{eq:thmR-window-law-convergence-NS}, the Portmanteau theorem,
the lower semicontinuity of
\[
x\mapsto
\norm{x(t)}^2+\norm{x(t)}_V^2+\norm{x(t)}^{2p},
\]
and \eqref{eq:complete-uniform-bound-NS},
\[
\begin{aligned}
\E\left(
\norm{u_*(t)}^2
+
\norm{u_*(t)}_V^2
+
\norm{u_*(t)}^{2p}
\right)
&\le
\liminf_{n\to\infty}
\E\left(
\norm{u_{\var_n}(t)}^2
+
\norm{u_{\var_n}(t)}_V^2
+
\norm{u_{\var_n}(t)}^{2p}
\right)\\
&\le C.
\end{aligned}
\]
Hence
\[
\sup_{t\in\R}
\E\left(
\norm{u_*(t)}^2
+
\norm{u_*(t)}_V^2
+
\norm{u_*(t)}^{2p}
\right)
<\infty.
\]

As \(\var_n\to0\) was arbitrary, every sequence admits a subsequence
satisfying \eqref{eq:thmR-complete-path-law-convergence-NS}.
The further extractions leading to
\eqref{eq:thmR-window-law-convergence-NS} preserve the same limit,
so every limit point is a bounded complete variational solution law
of \eqref{eq:SPDEtwo}.
\end{proof}

\begin{rem}
\label{cor:weak-second-unique-NS}
In Theorem \ref{thmR}, if the bounded complete variational solution
law $\Law{\bar u}$ of \eqref{eq:SPDEtwo} satisfying
\eqref{eq:thmR-averaged-bound-NS} is unique, then
\[
\Law{u_\var}\Rightarrow\Law{\bar u}
\quad\text{in }\cal P(C_{\mathrm{loc}}(\R;H)),
\qquad \var\to0.
\]
However, this uniqueness is not implied by the assumptions of
Theorem \ref{thmR}; a counterexample is given in Example \ref{no-uniqueness}.
\end{rem}

\iffalse
\begin{cor}
\label{cor:weak-second-unique-NS}
Assume that \ref{item:H-f}, \ref{item:H-g}, \ref{item:H-a}, and
\ref{item:H-d} hold, with \(2\nu\lambda_1>c_1+c_2\).
If \eqref{eq:SPDEtwo} admits a unique bounded complete variational 
solution law \(\Law{\bar u(\cdot)}\) satisfying, for some \(p>1\),
\[
\sup_{t\in\R}
\E\left(
\norm{\bar u(t)}^2
+
\norm{\bar u(t)}_V^2
+
\norm{\bar u(t)}^{2p}
\right)
<\infty,
\]
then
\[
\lim_{\var\to0}
d_{BL}
\left(
\Law{u_\var},
\Law{\bar u}
\right)
=0
\]
in \(\cal P(C_{\mathrm{loc}}(\R;H))\).
\end{cor}

\begin{proof}
Suppose that the conclusion is false. Then there exist
\(\eta_0>0\) and \(\var_n\to0\) such that
\[
d_{BL}
\left(
\Law{u_{\var_n}},
\Law{\bar u}
\right)
\ge\eta_0,
\qquad n\ge1.
\]
By Theorem \ref{thmR} (3), we can find a subsequence, still denoted by
\(\var_n\), and a bounded complete variational solution law $\Law{u_*(\cdot)}$ of
\eqref{eq:SPDEtwo} such that
\[
\Law{u_{\var_n}}
\Rightarrow
\Law{u_*}
\quad\text{in }\cal P(C_{\mathrm{loc}}(\R;H)).
\]
By uniqueness, \(\Law{u_*}=\Law{\bar u}\). Since \(d_{BL}\) metrizes
weak convergence on \(\cal P(C_{\mathrm{loc}}(\R;H))\),
\[
d_{BL}
\left(
\Law{u_{\var_n}},
\Law{\bar u}
\right)
\longrightarrow0,
\]
which contradicts the choice of \(\var_n\).
\end{proof}
\fi

\section{Global Averaging Principle in the Weak Sense}

In this section, we construct weak pullback attractors for all
variational solution laws and establish their upper semicontinuity
under averaging. For the general theory of pullback attractors for
nonautonomous dynamical systems, we refer to \cite{CLR2006}.

\subsection{Absorbing sets and weak attractors}

Solutions are understood in the sense of Definition
\ref{def:variational-martingale-solution-NS}.
For $F\in\cal H(F_0)$, equation \eqref{eq:SPDEone} with
coefficients $F$ is obtained by replacing $f$ and $g$
with the corresponding components of $F$.

Let $p>1$ be fixed as in
\eqref{eq:pullback-higher-moment-differential-NS}.
We denote by $\cal D_p$ the collection of all subsets
$K\subset\cal P_2(H)$ satisfying
\[
\sup_{\mu\in K}\int_H\norm{x}^{2p}\,\mu(dx)<\infty.
\]

\begin{de}
\label{def:pullback-absorbing-set-Dp-NS}
Fix \(0<\var\le1\).
A set \(B\subset\cal P_2(H)\) is called a \emph{pullback absorbing set}
for \eqref{eq:SPDEone} with respect to \(\cal D_p\) if, for every
\(K\in\cal D_p\) and \(F\in\cal H(F_0)\), there exists \(T_K>0\)
such that, for every \(t\ge T_K\), \(\mu\in K\), and every solution
\(u_\var\) on \([-t,0]\) with coefficients \(F\) and
\(\Law{u_\var(-t)}=\mu\),
\[\Law{u_\var(0)}\in B.\]

For \eqref{eq:SPDEtwo}, \(B\) is called an \emph{absorbing set}
with respect to \(\cal D_p\) if, for every \(K\in\cal D_p\), there
exists \(T_K>0\) such that, for every \(t\ge T_K\), \(\mu\in K\),
and every solution \(\bar u\) on \([0,t]\) with
\(\Law{\bar u(0)}=\mu\),
\[\Law{\bar u(t)}\in B.\]
\end{de}

\begin{de}
\label{def:weak-pullback-attractor-NS}
Fix $0<\var\le1$.
A family $\{\cal A^\var(F)\}_{F\in\cal H(F_0)}$ is called a
\emph{weak pullback attractor} for \eqref{eq:SPDEone}
with respect to $\cal D_p$ if:
\begin{itemize}
\item[(i)] For every $F\in\cal H(F_0)$,
$\cal A^\var(F)\subset\cal P_2(H)$ is nonempty and compact
in the $d_{BL}$-topology.
\item[(ii)] For every $F\in\cal H(F_0)$ and every nonempty
$K\in\cal D_p$,
\[
\lim_{t\to\infty}
\sup_{u_\var}
\inf_{\nu\in\cal A^\var(F)}
d_{BL}\bigl(\Law{u_\var(0)},\nu\bigr)=0,
\]
where the supremum is over all solutions $u_\var$ on $[-t,0]$
with coefficients $F$ and $\Law{u_\var(-t)}\in K$.
\item[(iii)] For every $F\in\cal H(F_0)$,
$\cal A^\var(F)$ is contained in every $d_{BL}$-closed subset
of $\cal P_2(H)$ satisfying \textup{(ii)} for this $F$.
\end{itemize}
For \eqref{eq:SPDEtwo}, a nonempty $d_{BL}$-compact set
$\bar{\cal A}\subset\cal P_2(H)$ is called a
\emph{weak global attractor} with respect to $\cal D_p$
if it is the smallest $d_{BL}$-closed set satisfying
\[
\lim_{t\to\infty}
\sup_{\bar u}
\inf_{\nu\in\bar{\cal A}}
d_{BL}\bigl(\Law{\bar u(t)},\nu\bigr)=0
\]
for every nonempty $K\in\cal D_p$.
Here the supremum is over all solutions $\bar u$ on $[0,t]$
with $\Law{\bar u(0)}\in K$.
\end{de}

\begin{lem}
\label{lem:pullback-absorbing-set-NS}
Assume that \ref{item:H-f}, \ref{item:H-g}(i), and
\ref{item:H-d} hold with $2\nu\lambda_1>c_1+c_2$.
Then, for every $0<\var\le1$, equation \eqref{eq:SPDEone}
admits a pullback absorbing set $B_{R,p}$ with respect to
$\cal D_p$.

If \ref{item:H-a} also holds, the same set is absorbing
for \eqref{eq:SPDEtwo} with respect to $\cal D_p$.
\end{lem}

\begin{proof}
By \eqref{eq:pullback-higher-moment-differential-NS}, set
\[
\alpha_p
:=
p\left[
2\nu\lambda_1-c_1-c_2-4(p-1)C
-\eta\left(c_0+2(p-1)C\right)
\right]>0,
\qquad
C_p:=C_{p,\eta}.
\]
Choose \(R_p>C_p/\alpha_p\) and set
\[
B_{R,p}
:=
\left\{
\mu\in\cal P_2(H):
\int_H\norm{x}^{2p}\,\mu(dx)\le R_p
\right\}.
\]
For \(K\in\cal D_p\), set
\[
M_K
:=
\sup_{\mu\in K}
\int_H\norm{x}^{2p}\,\mu(dx)<\infty.
\]
Fix \(0<\var\le1\), \(F\in\cal H(F_0)\), \(t>0\), and
\(\mu\in K\), and let \(u_\var\) be any solution of
\eqref{eq:SPDEone} on \([-t,0]\) with coefficients \(F\) and
\(\Law{u_\var(-t)}=\mu\). 
The bounds in \ref{item:H-f}, \ref{item:H-g}(i), and
\ref{item:H-d} hold uniformly for every \(F\in\cal H(F_0)\).
Thus, applying the argument yielding
\eqref{eq:pullback-higher-moment-differential-NS} on \([-t,0]\)
and Gronwall's inequality, we obtain
\begin{align}
\label{97:11}
\E\norm{u_\var(0)}^{2p}
&\le
e^{-\alpha_p t}\E\norm{u_\var(-t)}^{2p}
+\frac{C_p}{\alpha_p}\notag\\
&=
e^{-\alpha_p t}
\int_H\norm{x}^{2p}\,\mu(dx)
+\frac{C_p}{\alpha_p}
\le
e^{-\alpha_p t}M_K+\frac{C_p}{\alpha_p}.
\end{align}
Choose \(T_K>0\) such that
\[e^{-\alpha_p T_K}M_K+C_p/\alpha_p\le R_p.\]
Then \eqref{97:11} implies
\(\Law{u_\var(0)}\in B_{R,p}\) for every \(t\ge T_K\).
Since \(\alpha_p\) and \(C_p\) are independent of \(\var\) and \(F\),
so is \(T_K\). Hence \(B_{R,p}\) is pullback absorbing in the sense of
Definition \ref{def:pullback-absorbing-set-Dp-NS}.

Under \ref{item:H-a}, Remarks
\ref{rem:averaged-coefficients-NS}(ii) and \ref{rem:bar}
give the same bounds for \(\bar f\) and \(\bar g\).
Applying the argument yielding \eqref{97:11} on \([0,t]\), we obtain
\[
\E\norm{\bar u(t)}^{2p}
\le
e^{-\alpha_p t}M_K+\frac{C_p}{\alpha_p}
\]
whenever \(\Law{\bar u(0)}\in K\). Thus
\(\Law{\bar u(t)}\in B_{R,p}\) for every \(t\ge T_K\), and
\(B_{R,p}\) is also absorbing for \eqref{eq:SPDEtwo}.
\end{proof}

\begin{lem}
\label{lem:pullback-asymptotic-compactness-NS}
Assume that \ref{item:H-f}, \ref{item:H-g}, and
\ref{item:H-d} hold with $2\nu\lambda_1>c_1+c_2$.
Fix $0<\var\le1$, $F\in\cal H(F_0)$, and $K\in\cal D_p$.
Let $t_n>0$ satisfy $t_n\to\infty$.
For each $n$, let $u_n$ be any solution of \eqref{eq:SPDEone}
on $[0,t_n]$ with coefficients $\sigma_{-t_n/\var}F$ and
$\Law{u_n(0)}=\mu_n\in K$.
Then $\{\Law{u_n(t_n)}:n\ge1\}$ is relatively compact in
$(\cal P_2(H),W_2)$.

If \ref{item:H-a} also holds, the same conclusion holds
for the averaged equation \eqref{eq:SPDEtwo}.
\end{lem}

\begin{proof}
Removing finitely many terms and relabeling if necessary,
we may assume $t_n\ge1$ for every $n\ge1$.
Set 
\[M_K:=\sup_{\mu\in K}\int_H\norm{x}^{2p}\,\mu(dx)<\infty.\]
By Lemma \ref{lem:energy-dissipativity-NS},
\[
\sup_{n\ge1}\sup_{0\le r\le t_n}\E\norm{u_n(r)}^2\le C_K.
\]
Integrating \eqref{37:2} over \([t_n-1,t_n]\) gives
\[
\sup_{n\ge1}\int_{t_n-1}^{t_n}\E\norm{u_n(r)}_V^2\,dr\le C_K.
\]
Hence, there exists \(s_n\in[t_n-1,t_n]\) such that
\[\E\norm{u_n(s_n)}_V^2\le C_K.\] 
Applying Lemma \ref{lemtwo} on \([s_n,t_n]\), we obtain
\begin{equation}
\label{eq:pullback-compactness-terminal-V-bound-NS}
\sup_{n\ge1}\E\norm{u_n(t_n)}_V^2<\infty.
\end{equation}
Applying the argument yielding \eqref{97:11} on $[0,t_n]$
with coefficients $\sigma_{-t_n/\var}F$, we have
\[
\E\norm{u_n(t_n)}^{2p}
\le
e^{-\alpha_p t_n}
\int_H\norm{x}^{2p}\,\mu_n(dx)
+\frac{C_p}{\alpha_p}
\le
e^{-\alpha_p t_n}M_K+\frac{C_p}{\alpha_p}.
\]
Thus,
\begin{equation}
\label{eq:pullback-compactness-terminal-H2p-bound-NS}
\sup_{n\ge1}\E\norm{u_n(t_n)}^{2p}<\infty.
\end{equation}
By \eqref{eq:pullback-compactness-terminal-V-bound-NS},
\eqref{eq:pullback-compactness-terminal-H2p-bound-NS}, and
the compactness argument yielding \eqref{eq:initial-W2-limit-NS},
$\{\Law{u_n(t_n)}:n\ge1\}$ is relatively compact in
$(\cal P_2(H),W_2)$.

Under \ref{item:H-a}, Remarks
\ref{rem:averaged-coefficients-NS}(ii) and \ref{rem:bar}
give the same coefficient bounds for $\bar f,\bar g$.
The arguments yielding
\eqref{eq:pullback-compactness-terminal-V-bound-NS},
\eqref{eq:pullback-compactness-terminal-H2p-bound-NS}, and
\eqref{eq:initial-W2-limit-NS} therefore apply to
\eqref{eq:SPDEtwo}.
\end{proof}

\begin{prop}
\label{thm:weak-pullback-attractor-existence-NS}
Assume that \ref{item:H-f}, \ref{item:H-g}, and
\ref{item:H-d} hold with $2\nu\lambda_1>c_1+c_2$.
For every $0<\var\le1$, equation \eqref{eq:SPDEone} admits a
weak pullback attractor
$\{\cal A^\var(F)\}_{F\in\cal H(F_0)}$
with respect to $\cal D_p$.
More precisely,
\begin{equation}
\label{eq:NS-weak-pullback-attractor-definition}
\cal A^\var(F)
:=
\bigcap_{\tau>0}
\overline{
\bigcup_{t\ge\tau}
\left\{
\Law{u_\var(0)}:
\Law{u_\var(-t)}\in B_{R,p}
\right\}
}^{\,d_{BL}},
\qquad F\in\cal H(F_0),
\end{equation}
where, for each \(t\ge\tau\), the set is taken over all
solutions \(u_\var\) of \eqref{eq:SPDEone} on \([-t,0]\) with
coefficients \(F\), and \(B_{R,p}\) is defined in
Lemma \ref{lem:pullback-absorbing-set-NS}.
\end{prop}

\begin{proof}
Fix \(0<\var\le1\) and \(F\in\cal H(F_0)\). By
\eqref{eq:NS-weak-pullback-attractor-definition},
\(\nu\in\cal A^\var(F)\) if and only if there exist \(t_n\to\infty\)
and solutions \(u_n\) on \([-t_n,0]\) with coefficients \(F\) and
\(\Law{u_n(-t_n)}\in B_{R,p}\) such that
\begin{equation}
\label{eq:NS-weak-attractor-characterization}
d_{BL}\bigl(\Law{u_n(0)},\nu\bigr)\longrightarrow0.
\end{equation}
Indeed, if \(\nu\in\cal A^\var(F)\), then for every \(n\) we may choose
\(t_n\ge n\) and such a solution satisfying
\[d_{BL}(\Law{u_n(0)},\nu)<\frac{1}{n};\]
the converse follows from
\eqref{eq:NS-weak-pullback-attractor-definition}.

For each \(n\in\N\), choose a solution \(u_n\) on \([-n,0]\) with
coefficients \(F\) and \(\Law{u_n(-n)}=\delta_0\). By Lemma
\ref{lem:pullback-asymptotic-compactness-NS}, the sequence
\(\{\Law{u_n(0)}\}\) has a subsequence converging in \(W_2\).
Since \(\delta_0\in B_{R,p}\), \eqref{eq:NS-weak-attractor-characterization}
implies \(\cal A^\var(F)\ne\emptyset\). 

Let \(\nu_n\in\cal A^\var(F)\). By
\eqref{eq:NS-weak-attractor-characterization}, for every \(n\) there
exist \(t_n\ge n\) and a solution \(u_n\) on \([-t_n,0]\) with
coefficients \(F\) and \(\Law{u_n(-t_n)}\in B_{R,p}\) such that
\[d_{BL}(\Law{u_n(0)},\nu_n)\le\frac{1}{n}.\]
Lemma \ref{lem:pullback-asymptotic-compactness-NS} gives, after passing to a
subsequence,
\[
W_2\bigl(\Law{u_n(0)},\nu\bigr)\longrightarrow0
\]
for some \(\nu\in\cal P_2(H)\). Hence
\[
d_{BL}(\nu_n,\nu)
\le
\frac1n+d_{BL}\bigl(\Law{u_n(0)},\nu\bigr)
\longrightarrow0.
\]
Since \(t_n\to\infty\) and
\(\Law{u_n(-t_n)}\in B_{R,p}\),
\eqref{eq:NS-weak-attractor-characterization} gives
\(\nu\in\cal A^\var(F)\). Thus \(\cal A^\var(F)\) is compact in the
\(d_{BL}\)-topology.

Let \(K\in\cal D_p\) be nonempty. If pullback attraction fails, there
exist \(\eta_0>0\), \(t_n\to\infty\), and solutions \(u_n\) on
\([-t_n,0]\) with coefficients \(F\) and
\(\Law{u_n(-t_n)}\in K\) such that
\begin{equation}
\label{eq:NS-weak-attraction-contradiction}
\inf_{\nu\in\cal A^\var(F)}
d_{BL}\bigl(\Law{u_n(0)},\nu\bigr)\ge\eta_0.
\end{equation}
By Lemma \ref{lem:pullback-asymptotic-compactness-NS}, after passing
to a subsequence,
\begin{equation}
\label{98:1}
W_2\bigl(\Law{u_n(0)},\nu_0\bigr)\longrightarrow0
\end{equation}
for some \(\nu_0\in\cal P_2(H)\). Let \(T_K\) be given by Lemma
\ref{lem:pullback-absorbing-set-NS}. 
For sufficiently large \(n\), \(t_n>T_K\). By Lemma
\ref{lem:pullback-absorbing-set-NS},
\[
\Law{u_n(-t_n+T_K)}\in B_{R,p},
\]
since the translated coefficients still belong to
\(\cal H(F_0)\) and \(T_K\) is independent of \(F\).
As \(t_n-T_K\to\infty\), by \eqref{eq:NS-weak-attractor-characterization}
and \eqref{98:1}, we have 
\(\nu_0\in\cal A^\var(F)\), contradicting
\eqref{eq:NS-weak-attraction-contradiction}.
Therefore,
\(\cal A^\var(F)\) pullback attracts every \(K\in\cal D_p\).

Finally, every \(d_{BL}\)-closed set satisfying Definition
\ref{def:weak-pullback-attractor-NS}\textup{(ii)} for \(F\) attracts
\(B_{R,p}\), since \(B_{R,p}\in\cal D_p\). For any
\(\nu\in\cal A^\var(F)\), choose \(t_n\to\infty\) and \(u_n\) as in
\eqref{eq:NS-weak-attractor-characterization}. By the attraction of
\(B_{R,p}\), 
\(d_{BL}(\Law{u_n(0)},\nu)\to0\) and the \(d_{BL}\)-closedness,
we obtain that this set contains \(\nu\). Thus it contains
\(\cal A^\var(F)\), which proves the minimality.
\end{proof}

\begin{rem}
\label{A-aver}
Assume that \ref{item:H-f}--\ref{item:H-d} hold with
\(2\nu\lambda_1>c_1+c_2\).
By Remarks \ref{rem:averaged-coefficients-NS}(ii) and
\ref{rem:bar}, Proposition
\ref{thm:weak-pullback-attractor-existence-NS} applies to
\eqref{eq:SPDEtwo}. Hence \eqref{eq:SPDEtwo} admits a weak global
attractor with respect to \(\cal D_p\), given by
\[
\bar{\cal A}
:=
\bigcap_{\tau>0}
\overline{
\bigcup_{t\ge\tau}
\left\{
\Law{\bar u(t)}:
\Law{\bar u(0)}\in B_{R,p}
\right\}
}^{\,d_{BL}},
\]
where, for each \(t\ge\tau\), the set is taken over all solutions
\(\bar u\) of \eqref{eq:SPDEtwo} on \([0,t]\), and the closure is
taken in \(\cal P_2(H)\). Moreover, \(\bar{\cal A}\) attracts every
\(K\in\cal D_p\). In particular,
\[
\lim_{t\to\infty}
\sup_{\bar u}
\inf_{\nu\in\bar{\cal A}}
d_{BL}\bigl(\Law{\bar u(t)},\nu\bigr)=0,
\]
where the supremum is taken over all solutions \(\bar u\) on
\([0,t]\) with \(\Law{\bar u(0)}\in B_{R,p}\).
\end{rem}

\begin{proof}[Proof of Theorem
\ref{thm:global-weak-attractor-averaging-NS}]
Assertions \textup{(1)} and \textup{(2)} follow from Proposition
\ref{thm:weak-pullback-attractor-existence-NS} and Remark
\ref{A-aver}.

Suppose that assertion \textup{(3)} is false. Then there exist
\(\eta_0>0\), \(\var_n\to0\), \(F_n\in\cal H(F_0)\), and
\(\nu_n\in\cal A^{\var_n}(F_n)\) such that
\[
\inf_{\nu\in\bar{\cal A}}d_{BL}(\nu_n,\nu)\ge\eta_0,
\qquad n\ge1.
\]
By Remark \ref{A-aver}, we choose \(T>0\) such that
\[
\sup_{\bar u}\inf_{\nu\in\bar{\cal A}}
d_{BL}\bigl(\Law{\bar u(T)},\nu\bigr)<\frac{\eta_0}{2},
\]
where the supremum is taken over all solutions of
\eqref{eq:SPDEtwo} on \([0,T]\) with
\(\Law{\bar u(0)}\in B_{R,p}\).

Let \(T_K\) be the absorbing time for \(K=B_{R,p}\) given by
Lemma \ref{lem:pullback-absorbing-set-NS}. From
\eqref{eq:NS-weak-attractor-characterization}, for every \(n\) we
choose \(t_n\ge T+T_K+n\) and a solution \(u_n\) of
\eqref{eq:SPDEone} on \([-t_n,0]\) with coefficients \(F_n\) such that
\begin{align}\label{914:1}
\Law{u_n(-t_n)}\in B_{R,p},
\qquad
d_{BL}\bigl(\Law{u_n(0)},\nu_n\bigr)\le\frac1n.
\end{align}
As \(t_n-T\ge T_K\) and
\(\sigma_{-T/\var_n}F_n\in\cal H(F_0)\), Lemma
\ref{lem:pullback-absorbing-set-NS} yields
\begin{align}\label{914:3}
\Law{u_n(-T)}\in B_{R,p}.
\end{align}
Applying Lemma \ref{lem:energy-dissipativity-NS} and the argument
yielding \eqref{eq:pullback-compactness-terminal-V-bound-NS}, we obtain
\[
\sup_{n\ge1}\E\norm{u_n(-T)}_V^2<\infty.
\]
Together with \eqref{914:3}, the argument yielding
\eqref{eq:initial-W2-limit-NS} shows that 
\(\overline{\{\Law{u_n(-T)}:n\ge1\}}^{W_2}\) is compact. By weak lower semicontinuity,
\[
\int_H\norm{x}^{2p}\,\mu(dx)\le R_p,\quad\text{for every } \mu\in\overline{\{\Law{u_n(-T)}:n\ge1\}}^{W_2}.
\]

By Remark \ref{rem:shift-uniform-first}(iii),
\(\sigma_{-T/\var_n}F_n\) satisfies
\ref{item:H-f}--\ref{item:H-d}. We may therefore
apply the arguments yielding
\eqref{eq:compact-averaging-uniform-estimates-NS} and
\eqref{eq:compact-averaging-residual-limit-NS}, together with the
limit identification leading to
\eqref{eq:thmT-limit-covariations-NS}, to obtain, after passing to a
subsequence,
\[
\Law{u_n(-T+\cdot)}
\Rightarrow
\Law{\bar u}
\quad\text{in }\cal P(C([0,T];H)),
\]
where \(\bar u\) is a variational martingale solution of
\eqref{eq:SPDEtwo}. Since \(\Law{u_n(-T)}\in B_{R,p}\), by continuity
of evaluation at \(0\) and weak lower semicontinuity, we have
\[\Law{\bar u(0)}\in B_{R,p}.\]

Evaluating at \(T\), we obtain
\begin{align}\label{914:2}
d_{BL}\bigl(\Law{u_n(0)},\Law{\bar u(T)}\bigr)\longrightarrow0.
\end{align}
By \eqref{914:1} and \eqref{914:2}, we get
\[
\eta_0
\le
\limsup_{n\to\infty}
\inf_{\nu\in\bar{\cal A}}d_{BL}(\nu_n,\nu)
\le
\inf_{\nu\in\bar{\cal A}}
d_{BL}\bigl(\Law{\bar u(T)},\nu\bigr)
<
\frac{\eta_0}{2},
\]
a contradiction. Thus assertion \textup{(3)} holds.
\end{proof}

\appendix
\section{}

\begin{lem}
\label{lem:mild-solution}
Assume that \ref{item:H-f} and \ref{item:H-g} hold.
Let $p>1$ and $\mu_0\in\cal P_2(H)$ satisfy
\[\int_H\norm{x}^{2p}\,\mu_0(dx)<\infty.\]
For every $s\in\R$ and $0<\var\le1$, equation
\eqref{eq:SPDEone} admits a variational martingale solution
on $[s,\infty)$ with initial law $\mu_0$ at time $s$.

If \ref{item:H-a} also holds, the same conclusion holds for
\eqref{eq:SPDEtwo}.
\end{lem}

\begin{proof}
It suffices to consider $s=0$.
By \eqref{BzeroH}, \eqref{bloc}, \eqref{fcon}, \eqref{gcon},
and 
\[\norm{B(u,u)}_{V^*}\le C\norm{u}\norm{u}_V,\]
the coefficient hypotheses of
\cite[Proposition 3.6]{hong2024mckean} hold.
Choose $u_0$ with $\Law{u_0}=\mu_0$.
For each $n\in\N$, that proposition gives a solution $u_n$
on $[0,n]$ with initial law
$\Law{nu_0/(n+\norm{u_0})}$.
By Lemma \ref{lemone}, for every $T>0$,
\begin{equation}
\label{eq:existence-uniform-energy-NS}
\sup_{n\ge T}
\left[
\E\sup_{0\le t\le T}\norm{u_n(t)}^{2p}
+\E\int_0^T\norm{u_n(t)}_V^2\,dt
\right]
\le C_{T,p}\left(1+\E\norm{u_0}^{2p}\right).
\end{equation}
The bound for $B$ and \eqref{fcon} imply that the drift
parts are bounded in probability in $W^{1,2}(0,T;V^*)$.
By the Burkholder--Davis--Gundy inequality and \eqref{gcon},
the stochastic integral parts are bounded in probability in
$C^\gamma([0,T];H)$ for $\gamma\in(0,(p-1)/(2p))$.
Using the compactness argument in
\cite[Lemma 3.5]{hong2024mckean}, we can obtain tightness in
$L^2(0,T;H)\cap C([0,T];V^*)$.

After a joint Skorokhod representation with the Wiener processes,
a subsequence converges almost surely to $u_\var$ in this space.
By \eqref{eq:existence-uniform-energy-NS} and uniform integrability, we have
\begin{equation}
\label{eq:existence-state-law-convergence-NS}
\int_0^T
W_2^2\left(\Law{u_n(t)},\Law{u_\var(t)}\right)dt
\le
\E\int_0^T\norm{u_n(t)-u_\var(t)}^2\,dt
\longrightarrow0.
\end{equation}
By \ref{item:H-f}, \ref{item:H-g}(i),
\eqref{eq:thmT-nonlinear-continuity-NS}, and
\eqref{eq:existence-state-law-convergence-NS},
the argument used to get \eqref{eq:thmT-limit-covariations-NS}
identifies the limit, using $2p$-moments of the coordinate
martingales.
Weak lower semicontinuity and the localized variational
It\^o formula
\cite[Theorem 4.2.5]{liu2015stochastic}
give an $H$-continuous version satisfying
Definition \ref{def:variational-martingale-solution-NS},
with initial law $\mu_0$.
By a diagonal extraction over $T\in\N$, we get a solution on
$[0,\infty)$. Moreover, time translation gives the assertion for every $s\in\R$.
By \ref{item:H-a} and Remark
\ref{rem:averaged-coefficients-NS}(ii), we conclude that 
the lemma also holds for \eqref{eq:SPDEtwo}.
\end{proof}

We now give the proof of Lemma~\ref{lemtwo}. 

\begin{proof}[Proof of Lemma \ref{lemtwo}]
Let $\{e_k\}_{k\ge1}\subset D(A)$ be an orthonormal basis of $H$
consisting of eigenfunctions of the Stokes operator. Set
$H_n:=\mathrm{span}\{e_1,\dots,e_n\}$, and let $P_n:H\to H_n$
be the orthogonal projection. On the stochastic basis of $u_\var$,
consider
\[
\begin{cases}
\begin{aligned}
du_n(t)
&+\left[\nu Au_n(t)+P_nB(u_n(t),u_n(t))\right]dt\\
&=P_nf\left(\frac t\var,u_\var(t),\Law{u_\var(t)}\right)dt
+P_ng\left(\frac t\var,u_\var(t),\Law{u_\var(t)}\right)dW_t,
\end{aligned}\\
u_n(0)=P_nu_0.
\end{cases}
\]
By \eqref{eq:SPDEthree},
\[\Law{u_\var(t)}\in\cal P_2(V) \quad\text{for almost every } t\in[0,T].\]
Applying It\^o's formula to $\norm{u_n(t)}_V^2$, using
\eqref{BzeroV}, Young's inequality, and the same localization and
martingale estimates as in the derivation of \eqref{eq:SPDEthree}, we obtain
\begin{align}
\label{eq:V-Galerkin-regularity-NS}
&\E\sup_{0\le t\le T}\norm{u_n(t)}_V^2
+\E\int_0^T\norm{Au_n(t)}^2\,dt\notag\\
&\quad\le C\left[
\E\norm{u_0}_V^2+
\E\int_0^T\left(
\norm{f\left(\frac t\var,u_\var(t),\Law{u_\var(t)}\right)}^2
+\norm{g\left(\frac t\var,u_\var(t),\Law{u_\var(t)}\right)}_{L_2(U,V)}^2
\right)dt\right]\notag\\
&\quad\le C_T\left[
1+\E\norm{u_0}_V^2+\E\int_0^T\norm{u_\var(t)}_V^2\,dt
\right]
\le C_T\left(1+\E\norm{u_0}_V^2\right),
\end{align}
where the last two inequalities follow from \ref{item:H-f},
\ref{item:H-g}(ii), the Poincar\'e inequality, and
\eqref{eq:SPDEthree}. The constants in \eqref{eq:V-Galerkin-regularity-NS} are independent of $n$, $\var$,
and the choice of solution.

Subtracting the projection of \eqref{eq:SPDEone} from the equation
for $u_n$, then using
\eqref{bzero}, \eqref{bloc}, Young's inequality, and
$\norm{P_nu_\var}_V\le\norm{u_\var}_V$, we obtain
\begin{align}
\label{eq:V-Galerkin-comparison-NS}
&\frac{d}{dt}\norm{u_n(t)-P_nu_\var(t)}^2
+\nu\norm{u_n(t)-P_nu_\var(t)}_V^2\notag\\
&\quad\le C\norm{u_\var(t)}_V^2
\norm{u_n(t)-P_nu_\var(t)}^2
+C\norm{
B(P_nu_\var(t),P_nu_\var(t))-B(u_\var(t),u_\var(t))
}_{V^*}^2.
\end{align}
Since $u_\var\in C([0,T];H)$ almost surely,
$P_nu_\var\to u_\var$ in $C([0,T];H)$ almost surely. By
\eqref{bskew}, the two-dimensional interpolation inequality, and
\eqref{eq:SPDEthree},
\begin{align}
\label{912:1}
&\int_0^T\norm{
B(P_nu_\var(t),P_nu_\var(t))-B(u_\var(t),u_\var(t))
}_{V^*}^2\,dt\notag\\
&\quad\le C\sup_{0\le t\le T}\norm{P_nu_\var(t)-u_\var(t)}
\sup_{0\le t\le T}\norm{u_\var(t)}
\int_0^T\norm{u_\var(t)}_V^2\,dt
\longrightarrow0
\quad\text{almost surely}.
\end{align}
By $u_n(0)=P_nu_0$, \eqref{eq:V-Galerkin-comparison-NS},
\eqref{912:1}, and Gronwall's inequality,
\[u_n\to u_\var \quad\text{in } C([0,T];H) \quad\text{almost surely}.\]
Thus, \eqref{eq:V-Galerkin-regularity-NS}, weak lower
semicontinuity, and Fatou's lemma give
\begin{equation}
\label{eq:V-basic-regularity-NS}
\E\sup_{0\le t\le T}\norm{u_\var(t)}_V^2
+\E\int_0^T\norm{Au_\var(t)}^2\,dt
\le C_T\left(1+\E\norm{u_0}_V^2\right).
\end{equation}
By \eqref{eq:V-basic-regularity-NS}, \ref{item:H-f}, and
\[\norm{B(u_\var,u_\var)}^2
\le C\norm{u_\var}_V^3\norm{Au_\var},\]
the drift of \eqref{eq:SPDEone} belongs to $L^2(0,T;H)$ almost surely.
Using \ref{item:H-g}(ii) and applying the localized variational
It\^o formula to $A^{1/2}u_\var$, we get a version of
$u_\var$ in $C([0,T];V)$.

By Young's inequality, \ref{item:H-f}, \ref{item:H-g}(ii), and
the Poincar\'e inequality,
\begin{align}
2\inpro{
f\left(\frac t\var,u_\var(t),\Law{u_\var(t)}\right),Au_\var(t)}
&\le\frac{\nu}{2}\norm{Au_\var(t)}^2
+C\left(1+\norm{u_\var(t)}^2+\E\norm{u_\var(t)}^2\right)\notag\\
&\le\frac{\nu}{2}\norm{Au_\var(t)}^2
+C\left(1+\norm{u_\var(t)}_V^2+\E\norm{u_\var(t)}_V^2\right),
\label{32:2}\\
\norm{g\left(\frac t\var,u_\var(t),\Law{u_\var(t)}\right)}_{L_2(U,V)}^2
&\le C\left(1+\norm{u_\var(t)}_V^2+\E\norm{u_\var(t)}_V^2\right).
\label{32:3}
\end{align}
Applying It\^o's formula to $\norm{u_\var(t)}_V^{2p}$ and using
\eqref{BzeroV}, \eqref{32:2}, and \eqref{32:3}, we obtain
\begin{align}
\label{eq:V-energy-differential-NS}
&d\norm{u_\var(t)}_V^{2p}
+p\nu\norm{u_\var(t)}_V^{2p-2}\norm{Au_\var(t)}^2\,dt\notag\\
&\quad\le C_p\left(
1+\norm{u_\var(t)}_V^2+\E\norm{u_\var(t)}_V^2
\right)\norm{u_\var(t)}_V^{2p-2}\,dt\notag\\
&\qquad+
2p\norm{u_\var(t)}_V^{2p-2}
\inpro{
g\left(\frac t\var,u_\var(t),\Law{u_\var(t)}\right)dW_t,
u_\var(t)}_V.
\end{align}
Repeating the localization, martingale, and Gronwall estimates used
to derive \eqref{eq:SPDEthree} from \eqref{HitoNSineq}, and using
\eqref{eq:V-basic-regularity-NS} and Jensen's inequality, we obtain
\begin{align}
\label{eq:V-solution-estimate-NS}
&\E\sup_{0\le t\le T}\norm{u_\var(t)}_V^{2p}
+\E\int_0^T
\norm{u_\var(t)}_V^{2p-2}\norm{Au_\var(t)}^2\,dt\notag\\
&\quad\le C_{T,p}\left[
1+\E\norm{u_0}_V^{2p}
+\int_0^T\left(\E\norm{u_\var(t)}_V^2\right)^p\,dt
\right]
\le C_{T,p}\left(1+\E\norm{u_0}_V^{2p}\right).
\end{align}
Under \ref{item:H-a}, Remark
\ref{rem:averaged-coefficients-NS}(ii) gives the same coefficient
bounds for $\bar f$ and $\bar g$, so the derivation of
\eqref{eq:V-solution-estimate-NS} also applies to $\bar u$.
\end{proof}

\begin{exam}
\label{no-uniqueness}
Let \(e_1\in D(A)\) be a normalized Fourier eigenfunction of the
Stokes operator whose second component vanishes and which depends
only on \(x_2\).
Then
\[
\norm{e_1}=1,
\qquad
(e_1\cdot\nabla)e_1=0,
\qquad
B(e_1,e_1)=0.
\]
Define
\[
f(t,u,\mu)=\bar f(u,\mu)
:=
\nu\sin\left(\frac{\pi}{2}\inpro{u,e_1}\right)Ae_1,
\qquad
g(t,u,\mu)=\bar g(u,\mu):=0.
\]
Since \(f(t,0,\delta_0)=0\) and
\[
\begin{aligned}
\norm{f(t,u,\mu_1)-f(t,v,\mu_2)}
&\le
\frac{\pi\nu}{2}\norm{Ae_1}
\left|\inpro{u-v,e_1}\right|\\
&\le
\frac{\pi\nu}{2}\norm{Ae_1}\norm{u-v},
\end{aligned}
\]
condition \ref{item:H-f} is satisfied. Moreover,
\ref{item:H-g} holds with \(L_g=L_g'=0\), while the time independence
of \(f\) and \(g\) gives \ref{item:H-a} with
\(\omega^f=\omega^g=0\). Finally,
\[
2\inpro{f(t,u,\mu),u}
\le
2\nu\norm{Ae_1}\norm{u}
\le
\nu\lambda_1\norm{u}^2
+\frac{\nu}{\lambda_1}\norm{Ae_1}^2.
\]
Hence \ref{item:H-d} holds with
\[
c_0=\frac{\nu}{\lambda_1}\norm{Ae_1}^2,
\qquad
c_1=\nu\lambda_1,
\qquad
c_2=0,
\]
and therefore \(2\nu\lambda_1>c_1+c_2\).

Since
\[
\bar f(0,\mu)=0,
\qquad
\bar f(e_1,\mu)=\nu Ae_1,
\qquad
\bar f(-e_1,\mu)=-\nu Ae_1,
\]
and \(B(e_1,e_1)=B(-e_1,-e_1)=0\), the three constant processes
\[
\bar u(t)\equiv0,
\qquad
\bar u(t)\equiv e_1,
\qquad
\bar u(t)\equiv-e_1,
\qquad t\in\R,
\]
satisfy the variational identity for \eqref{eq:SPDEtwo}. Since
\(\bar g=0\), each of them is a complete variational martingale
solution of \eqref{eq:SPDEtwo} on any stochastic basis carrying a
cylindrical Wiener process on \(U\). For every \(p>1\),
\[
\sup_{t\in\R}
\E\left(
\norm{\bar u(t)}^2+
\norm{\bar u(t)}_V^2+
\norm{\bar u(t)}^{2p}
\right)
\le
2+\norm{e_1}_V^2<\infty.
\]
At \(t=0\), the three solutions have laws
\(\delta_0\), \(\delta_{e_1}\), and \(\delta_{-e_1}\), respectively.
Hence their complete solution laws are distinct, and
\eqref{eq:SPDEtwo} admits at least three bounded complete variational
solution laws satisfying \eqref{eq:thmR-averaged-bound-NS}.
\end{exam}

\section*{Acknowledgements}
This work is supported by National Key R\&D Program of China (No. 2023YFA1009200) and 
NSFC (Grants 12531009 and 11925102).

\end{document}